\documentclass[10pt,a4paper]{amsart}
\usepackage[utf8]{inputenc}
\usepackage{amsmath,amssymb,amsthm,enumerate,ulem,color}
\usepackage{hyperref} 
\usepackage[shortlabels]{enumitem}
\usepackage{hyperref}
\newtheorem{theorem}{Theorem}[section]
\newtheorem{thmx}{Theorem}

\newtheorem{proposition}[theorem]{Proposition}
\newtheorem{corollary}[theorem]{Corollary}

\newtheorem{lemma}[theorem]{Lemma}
\newtheorem{fact}[theorem]{Fact}
\theoremstyle{definition}
\newtheorem{definition}[theorem]{Definition}
\newtheorem{notation}[theorem]{Notation}
\newtheorem*{thm*}{Theorem}

\newtheorem*{claim*}{Claim}

\newtheorem{remark}[theorem]{Remark}

\newtheorem*{convention}{Convention}
\newcounter{que}
\newtheorem{question}[que]{Question}

\newcommand{\setsep}{:\;}

\newcommand\isomtrclass[2][]{\langle #2\rangle_{\equiv}^{#1}}
\newcommand\isomrfclass[2][]{\langle #2\rangle_{\simeq}^{#1}}
\newcommand\approximates[2][]{\stackrel{#2}{\sim} #1}
\newcommand{\Span}{\operatorname{span}}

\def\Gurarii{\mathbb{G}}

\def\F{\mathcal{F}}

\def\B{\mathcal{B}}

\def\O{\mathcal{O}}

\def\PP{\mathcal{P}}

\def\T{\mathcal{T}}
\newcommand{\Rat}{\mathbb{Q}}
\newcommand{\Rea}{\mathbb{R}}
\newcommand{\Nat}{\mathbb{N}}

\newcommand{\qeSpan}{\Span_\Rat}
\def\b{\mathfrak{b}}
\def \dmu {\, \mathrm{d} \mu}

\begin{document}
\normalem 

\title[Polish spaces of Banach spaces]{Polish spaces of Banach spaces.\\ Complexity of isometry and isomorphism classes}

\author[M. C\' uth]{Marek C\'uth}
\author[M. Dole\v{z}al]{Martin Dole\v{z}al}
\author[M. Doucha]{Michal Doucha}
\author[O. Kurka]{Ond\v{r}ej Kurka}
\email{cuth@karlin.mff.cuni.cz}
\email{dolezal@math.cas.cz}
\email{doucha@math.cas.cz}
\email{kurka.ondrej@seznam.cz}

\address[M.~C\' uth]{Charles University, Faculty of Mathematics and Physics, Department of Mathematical Analysis, Sokolovsk\'a 83, 186 75 Prague 8, Czech Republic}
\address[M.~Dole\v{z}al, M.~Doucha, O.~Kurka]{Institute of Mathematics of the Czech Academy of Sciences, \v{Z}itn\'a 25, 115 67 Prague 1, Czech Republic}


\thanks{M. C\'uth was supported by Charles University Research program No. UNCE/SCI/023. Research of Martin Dole\v{z}al was supported by the GA\v{C}R project 20-22230L and RVO: 67985840. M. Doucha and O. Kurka were supported by the GA\v{C}R project 22-07833K and RVO: 67985840.
}

\begin{abstract}
   We study the complexities of isometry and isomorphism classes of separable Banach spaces in the Polish spaces of Banach spaces recently introduced and investigated by the authors in \cite{CDDK1}. We obtain sharp results concerning the most classical separable Banach spaces.
    
    We prove that the infinite-dimensional separable Hilbert space is characterized as the unique separable infinite-dimensional Banach space whose isometry class is closed, and also as the unique separable infinite-dimensional Banach space whose isomorphism class is $F_\sigma$. For $p\in\left[1,2\right)\cup\left(2,\infty\right)$, we show that the isometry classes of $L_p[0,1]$ and $\ell_p$ are $G_\delta$-complete sets and $F_{\sigma\delta}$-complete sets, respectively. Then we show that the isometry class of $c_0$ is an $F_{\sigma\delta}$-complete set.
     
     Additionally, we compute the complexities of many other natural classes of separable Banach spaces; for instance, the class of separable $\mathcal{L}_{p,\lambda+}$-spaces, for $p,\lambda\geq 1$, is shown to be a $G_\delta$-set, the class of superreflexive spaces is shown to be an $F_{\sigma\delta}$-set, and the class of spaces with local $\Pi$-basis structure is shown to be a $\boldsymbol{\Sigma}^0_6$-set. The paper is concluded with many open problems and suggestions for a future research.
\end{abstract}
\maketitle



\section*{Introduction}
Descriptive set theoretic approach to Banach spaces has proved to be a powerful tool in solving many problems in Banach space theory; for a wide selection of references ranging from the earliest ones of Bourgain to the most recent ones see e.g. \cite{Bourgain1, ArgDod, AGR03, ABB21, Cu18}. Traditionally, and as defined explicitly for the first time in the seminal papers of Bossard (\cite{Bossard93}, \cite{Bo02}), one considers the standard Borel space of all separable Banach spaces, which can be defined as an appropriate Borel subspace of the Effros-Borel space of all closed subspaces of some isometrically universal separable Banach space. Since such defined space of Banach spaces is not a topological space, it allows us to study Banach spaces globally only in a Borel way and not topologically, which would be desirable in some cases.  This drawback was addressed in a recent paper \cite{GS18} of Godefroy and Saint-Raymond where they propose to study certain natural topologies on the standard Borel spaces of separable Banach spaces and they compute Borel complexities, in these topologies, of several important classes of Banach spaces. The authors of this paper initiated in \cite{CDDK1} the study of the Polish spaces of norms and pseudonorms on the countable infinite-dimensional vector space over $\Rat$, which have additional further advantages:
\begin{itemize}
    \item they are almost canonical Polish spaces of separable Banach spaces;
    \item they have very nice topological properties that are connected to local theory of Banach spaces, especially to finite representability;
    \item the computation of Borel complexities of various classes of separable Banach spaces is usually in these spaces as straightforward as possible, which in particular allows us to improve several estimates by Godefroy and Saint-Raymond;
    \item this approach to topologizing the space of Banach spaces is somewhat similar to how metric structures are topologized generally in continuous model theory, which could connect descriptive set theory of Banach spaces with the model theory of Banach spaces, which is one of the most developed areas of applications of logic to metric structures.
\end{itemize} 

This paper is a companion and second part to \cite{CDDK1}, however it is completely self-contained and can be read independently. Its aim is to demonstrate the strength of this new approach by computing (in many cases, these computations are sharp) complexities of several classes of Banach spaces, focusing on isomorphism classes, improving several results of Godefroy and Saint-Raymond, and initiating the research of computing the complexities of isometry classes.

First we focus on the undoubtedly most important separable infinite-dimensional Banach space -- the Hilbert space $\ell_2(\Nat)$. We uniquely characterize it by both the complexity of its isometry class as well as the complexity of its isomorphism class.
\begin{thmx}\label{thm:Intro6}
\begin{enumerate}
    \item The separable infinite-dimensional Hilbert space is characterized as the unique separable infinite-dimensional Banach space whose isometry class is closed (see Theorem~\ref{thm:isometryClosedClass}).
    \item The separable infinite-dimensional Hilbert space is characterized as the unique, up to isomorphism, separable infinite-dimensional Banach space whose isomorphism class is $F_\sigma$ (see Theorem~\ref{thm:FSigmaIsomrfClasses}).
\end{enumerate}
\end{thmx}

\noindent Let us briefly comment on the importance of this result. There are many known isometric characterizations of inner product spaces and several of those may be easily seen to form a closed condition, see e.g. the monograph \cite{innerBook}. By our isometric characterization from Theorem~\ref{thm:Intro6}, there is no other separable infinite-dimensional Banach space which can be characterized by a closed condition. Our isomorphic characterization of Hilbert spaces from Theorem~\ref{thm:Intro6} seems to be even more interesting. Recall that Bossard \cite[Problem 2.9]{Bo02} originally asked whether $\ell_2$ is the unique space with a Borel isomorphism class. Although this is known to be false now (see e.g. \cite{G17}), our Theorem~\ref{thm:Intro6} shows that it is arguably the Banach space with the simplest possible isomorphism class. The only other candidate possibly having a simple isomorphism class is the Gurari\u{\i} space, which might potentially have a $G_\delta$ isomorphism class (see Proposition~\ref{prop:classOfl2}). We conjecture that it is not the case (in fact, we do not know whether the isomorphism class of the Gurari\u{\i} space is even Borel), however we cannot disprove it at the moment.

Moreover, since our coding of Banach spaces is connected to the local theory of Banach spaces, it is of some interest to notice that there were some attempts to characterize $\ell_2$ up to isomorphism using its finite-dimensional structure, see e.g. \cite[Conjecture 7.3]{JLS96} for the conjecture by Johnson, Lindenstrauss and Schechtman. Thus, our setting enables us to formulate and prove a result similar in a spirit to what was conjectured by Johnson, Lindenstrauss and Schechtman. We refer the reader to Section~\ref{sect:question} for more information in this direction.\medskip

Next, we continue by studying complexities of \emph{isometry} classes of other Banach spaces, i.e. how easy/difficult it is to define them uniquely up to isometry. There is an active ongoing research whether for a particular Banach space its isomorphism class is Borel or not (see e.g.  \cite{G17}, \cite{K19}, \cite{Gh16}, \cite{Godef17}) while it is known that isometry classes of separable Banach spaces are always Borel (note that the linear isometry relation is Borel bi-reducible with an orbit equivalence relation \cite{Melleray} and orbit equivalence relations have Borel equivalence classes \cite[Theorem 15.14]{KechrisBook}). Having a topology at our disposal we compute complexities of isometry classes of several classical Banach spaces.

\begin{thmx}\label{thm:Intro2}
\begin{enumerate}
    
    \item For $p\in[1,2)\cup(2,\infty)$, the isometry class of $L_p[0,1]$ is $G_\delta$-complete. Moreover, for every $\lambda\geq 1$, the class of separable $\mathcal{L}_{p,\lambda+}$-spaces is a $G_\delta$-set and the class of separable $\mathcal{L}_p$-spaces is a $G_{\delta \sigma}$-set, improving the estimate from \cite{GS18} (see Theorems~\ref{thm:classOfLp} and~\ref{thm:scriptLp}, and Corollary~\ref{cor:scriptLp}).
    
    \item For $p\in[1,2)\cup(2,\infty)$, the isometry class of $\ell_p$ is an $F_{\sigma \delta}$-complete set (see Theorem~\ref{thm:c0AndEllP}).
    
    \item The isometry class of $c_0$ is an $F_{\sigma \delta}$-complete set (see Theorem~\ref{thm:c0AndEllP}).
    
    \item The isometry class of the Gurari\u{\i} space is a $G_\delta$-complete set (see Corollary~\ref{cor:gurariiComplexity}).
\end{enumerate}
\end{thmx}

Let us also present here a few sample results which involve complexities of more general classes of Banach spaces.
\begin{thmx}\label{thm:Intro3}
\begin{enumerate}
    \item The class of all superreflexive spaces is an $F_{\sigma\delta}$-set (see Theorem~\ref{thm:superreflexive}).
    \item The class of all spaces with local $\Pi$-basis structure is a $\boldsymbol{\Sigma}^0_6$-set (see Theorem~\ref{thm:lbs}).
    \item For a fixed ordinal $\alpha\in[1,\omega_1)$, the class of spaces whose Szlenk index is bounded by $\omega^\alpha$ is a $\boldsymbol{\Pi}^0_{\omega^\alpha+1}$-set (see Theorem~\ref{thm:Szlenk}).
\end{enumerate}
\end{thmx}

\section{Preliminaries}\label{sect:pozustatky}

In this section we set up some notation that will be used throughout the paper and recall the notions and basic results from \cite{CDDK1} that we shall need in this paper.

\subsection{Notation}

Throughout the paper we usually denote the Borel classes of low complexity by the traditional notation such as $F_\sigma$ and $G_\delta$, or even $F_{\sigma \delta}$ (countable intersection of $F_\sigma$ sets) and $G_{\delta \sigma}$ (countable union of $G_\delta$-sets). However whenever it is more convenient or necessary we use the notation $\boldsymbol{\Sigma}_\alpha^0$, resp. $\boldsymbol{\Pi}_\alpha^0$, where $\alpha<\omega_1$ (we refer to \cite[Section 11]{KechrisBook} for this notation). We emphasize that open sets, resp. closed sets, are $\boldsymbol{\Sigma}^0_1$, resp. $\boldsymbol{\Pi}^0_1$, by this notation.

In a few occassions, for a Borel class $\boldsymbol{\Gamma}$ we will use the notion of $\boldsymbol{\Gamma}$-hard and $\boldsymbol{\Gamma}$-complete sets. We refer the reader to \cite[Definition 22.9]{KechrisBook} for  these notions. For a reader not familiar with them, let us emphasize that a set $A$ being $\boldsymbol{\Gamma}$-hard, for a Borel class $\boldsymbol{\Gamma}$, in particular implies that $A$ is not of a lower complexity than $\boldsymbol{\Gamma}$. Thus results stating that some set is $\boldsymbol{\Sigma}^0_\alpha$-complete means that the set is $\boldsymbol{\Sigma}^0_\alpha$ and not simpler.

Let us also state here the following simple lemma. Although it should be well known, we could not find a proper reference, so we provide a sketch of the proof.

\begin{lemma}\label{lem:F_sigma-hard}
Suppose that $X$ is a Polish space and $B\subseteq X$ is a Borel set which is not a $G_\delta$-set. Then $B$ is $F_\sigma$-hard. The same with the roles of $G_\delta$ and $F_\sigma$ interchanged.
\end{lemma}
\begin{proof}
By Hurewicz theorem (see e.g. \cite[Theorem 21.18]{KechrisBook}), there is a set $C\subseteq X$ homeomorphic to the Cantor space such that $C\cap B$ is countable dense in $C$.
Then $C\cap B$ is an $F_\sigma$-set but not a $G_\delta$-set in the zero-dimensional Polish space $C$, and so it is $F_\sigma$-complete in $C$ by Wadge's theorem (see e.g. \cite[Theorem 22.10]{KechrisBook}).
So for any zero-dimensional Polish space $Y$ and any $F_\sigma$-subset $A$ of $Y$, there is a Wadge reduction of $A\subseteq Y$ to $C\cap B\subseteq C$. But any such reduction is also a reduction of $A\subseteq Y$ to $B\subseteq X$, and so $B$ is $F_\sigma$-hard.

The argument with the roles of $F_\sigma$ and $G_\delta$ interchanged is similar.
\end{proof}

Moreover, given a class $\boldsymbol{\Gamma}$ of sets in metrizable spaces, we say that $f:X\to Y$ is $\boldsymbol{\Gamma}$-measurable if $f^{-1}(U)\in \boldsymbol{\Gamma}$ for every open set $U\subseteq Y$.\medskip

Given Banach spaces $X$ and $Y$, we denote by $X\equiv Y$ (resp. $X\simeq Y$) the fact that those two spaces are linearly isometric (resp. isomorphic). We denote by $X\hookrightarrow Y$ the fact that $Y$ contains a subspace isomorphic to $X$. For $K\geq 1$, a \emph{$K$-isomorphism} $T:X\to Y$ is a linear map with $K^{-1}\|x\|\leq\|Tx\|\leq K\|x\|$, $x\in X$. If $x_1,\ldots,x_n$ are linearly independent elements of $X$ and $y_1,\ldots,y_n\in Y$, we write $(Y,y_1,\ldots,y_n)\approximates[(X,x_1,\ldots,x_n)]{K}$ if the linear operator $T:\Span\{x_1,\ldots,x_n\}\to\Span\{y_1,\ldots,y_n\}$ sending $x_i$ to $y_i$ satisfies $\max\{\|T\|,\|T^{-1}\|\}<K$. If $X$ has a canonical basis $(x_1,\ldots,x_n)$ which is clear from the context, we just write $(Y,y_1,\ldots,y_n)\approximates[X]{K}$ instead of $(Y,y_1,\ldots,y_n)\approximates[(X,x_1,\ldots,x_n)]{K}$. Morevoer, if $Y$ is clear from the context we write $(y_1,\ldots,y_n)\approximates[X]{K}$ instead of $(Y,y_1,\ldots,y_n)\approximates[X]{K}$.

Throughout the text $\ell_p^n$ denotes the $n$-dimensional $\ell_p$-space, i.e. the upper index denotes dimension. 

Finally, in order to avoid any confusion, we emphasize that if we write that a mapping is an ``isometry'' or an ``isomorphism'', we do not mean it is surjective if this is not explicitly mentioned.

\subsection{Notions and results from \cite{CDDK1}}
The most important notion we want to recall is the Polish spaces of Banach spaces. We refer to \cite[Section 2]{CDDK1} for a proper introduction.
\bigskip

By $V$, let us denote the vector space over $\Rat$ of all finitely supported sequences of rational numbers; that is, the unique infinite-dimensional vector space over $\Rat$ with a countable Hamel basis $(e_n)_{n\in\Nat}$, which we may view as the vector space of all finitely supported rational sequences.

\begin{definition}{\cite[Definition 2.1]{CDDK1}}\label{def:basic}
Let us denote by $\PP$ the space of all pseudonorms on the vector space $V$. Since $\PP$ is a closed subset of $\Rea^V$, this gives $\PP$ the Polish topology inherited from $\Rea^V$. The subbasis of this topology is given by sets of the form $U[v,I]:=\{\mu\in\PP\setsep \mu(v)\in I\}$, where $v\in V$ and $I$ is an open interval.

We often identify $\mu\in\PP$ with its extension to the pseudonorm on the space $c_{00}$, that is, on the vector space over $\Rea$ of all finitely supported sequences of real numbers.

For every $\mu\in\PP$ we denote by $X_\mu$ the Banach space given as the completion of the quotient space $X/N$, where $X = (c_{00},\mu)$ and $N = \{x\in c_{00}\setsep \mu(x) = 0\}$. In what follows we often consider $V$ as a subspace of $X_\mu$, that is, we identify every $v\in V$ with its equivalence class $[v]_N\in X_\mu$.

By $\PP_\infty$ we denote the set of those $\mu\in\PP$ for which $X_\mu$ is infinite-dimensional Banach space, and by $\B$ we denote the set of those $\mu\in\PP_\infty$ for which the extension of $\mu$ to $c_{00}$ is an actual norm, that is, the vectors $ e_{1}, e_{2}, \dots $ are linearly independent in $X_\mu$.

We endow $\PP_\infty$ and $\B$ with topologies inherited from $\PP$.
\end{definition}

It is rather easy to verify that the topologies on $\PP_\infty$ and $\B$ are Polish, we refer to \cite[Corollary 2.5]{CDDK1} for a proof.

\begin{remark}\label{rem:admissible} As we have mentioned in the introduction, Godefroy and Saint-Raymond \cite{GS18} considered another way of topologizing the class of all separable (infinite-dimensional) Banach spaces. Namely, they consider the set $SB = \{X\subseteq C(2^\omega)\colon X \text{ is a closed linear subspace}\}$ and introduce a class of natural Polish topologies $\tau$ on $SB$ which they call ``admissible''. In \cite{CDDK1} we compared our spaces $\PP_\infty$ and $\B$ with $(SB,\tau)$. In particular, we observed that whenever $\tau$ is an admissible topology, then there exists a continuous map $\Phi:(SB,\tau)\to\PP$ such that for every $F\in SB(X)$ we have $F\equiv X_{\Phi(F)}$ isometrically, see \cite[Theorem 3.3]{CDDK1}. Thus, from our results obtained in the coding $\PP_\infty$ one may easily deduce also results formulated in the language of admissible typologies.
\end{remark}

The following definition precises the notation $\approximates{K}$ defined earlier.
\begin{definition}
If $v_1,\ldots,v_n\in V$ are given, for $\mu\in\PP$, instead of $(X_\mu,v_1,\ldots,v_n)\approximates[X]{K}$, we shall write $(\mu,v_1,\ldots,v_n)\approximates[X]{K}$.
\end{definition}
For further purposes, we record here the following lemma from \cite{CDDK1}.
\begin{lemma}[{\cite[Lemma 2.4]{CDDK1}}]\label{lem:infiniteDimIsGDelta}
Let $X$ be a Banach space with $\{x_1,\ldots,x_n\}\subseteq X$ linearly independent and let $v_1,\ldots,v_n\in V$. Then for any $K>1$ the set $$\mathcal{N}((x_i)_i,K,(v_i)_i)=\{\mu\in \PP\colon (\mu,v_1,\ldots,v_n)\approximates[(X,x_1,\ldots,x_n)]{K} \}$$ is open in $\PP$.

In particular, the set of those $\mu\in\PP$ for which the set $\{v_1,\ldots,v_n\}$ is linearly independent in $X_\mu$ is open in $\PP$.
\end{lemma}

Since we are interested mainly in subsets of $\PP$ closed under isometries, we introduce the following notation.

\begin{notation}\label{not:ismfrClass}
Let $Z$ be a separable Banach space and let $\mathcal{I}$ be a subset of $\PP$. We put
\[\isomtrclass[\mathcal{I}]{Z}:=\{\mu\in\mathcal{I}\setsep X_\mu\equiv Z\} \quad \text{and} \quad \isomrfclass[\mathcal{I}]{Z}:=\{\mu\in\mathcal{I}\setsep X_\mu\simeq Z\}.\]
If $\mathcal{I}$ is clear from the context we write $\isomtrclass{Z}$ and $\isomrfclass{Z}$ instead of $\isomtrclass[\mathcal{I}]{Z}$ and $\isomrfclass[\mathcal{I}]{Z}$ respectively.
\end{notation}

The connection of the topologies on $\PP$, $\PP_\infty$, and $\B$ with finite representability was thoroughly explored in \cite{CDDK1}. Here we recall what will be useful in this paper.

\begin{definition}
We say that a Banach space $X$ is finitely representable in a Banach space $Y$ if given any finite-dimensional subspace $E$ of $X$ and $\varepsilon > 0$ there exists a finite-dimensional subspace $F$ of $Y$ which is $(1+\varepsilon)$-isomorphic to $E$.

Moreover, if $\F$ is a family of Banach spaces, we say that a Banach space $X$ is finitely representable in $\F$ if given any finite-dimensional subspace $E$ of $X$ and any $\varepsilon > 0$ there exists a finite-dimensional subspace $F$ of some $Y\in\F$ which is $(1+\varepsilon)$-isomorphic to $E$.
\end{definition}

\begin{proposition}[{\cite[Proposition 2.9]{CDDK1}}]\label{prop:closureOfIsometryClasses}
If $X$ is a separable infinite-dimensional Banach space, then
    \[\{\nu\in\B\setsep X_\nu\text{ is finitely representable in }X\} =  \overline{\isomtrclass[\B]{X}}\cap\B\]
    and similarly also if we replace $\B$ with $\PP_\infty$ or with $\PP$.
    
Moreover, let $\F\subseteq \B$ be such that $\isomtrclass[\B]{X_\mu}\subseteq \F$ for every $\mu\in\F$. Then
\[\{\nu\in\B\setsep X_\nu\text{ is finitely representable in }\F\}= \overline{\F}\cap \B.\] The same again holds if we replace $\B$ with $\PP_\infty$ or with $\PP$.
\end{proposition}

We finish this section by recalling one particular Banach space that will play a fundamental role in certain further results in this paper. That is, we recall what the Gurari\u{\i} space is. One of the characterizations of the Gurari\u{\i} space is the following, for more details we refer the interested reader e.g. to \cite{CGK} (the characterization below is provided by \cite[Lemma 2.2]{CGK}).

\begin{definition}\label{def:Gurarii}
The Gurari\u{\i} space is the unique (up to isometry) separable Banach space such that  for every $\varepsilon > 0$ and every isometric embedding $g:A\to B$, where $B$ is a finite-dimensional Banach space and $A$ is a subspace of $\mathbb{G}$, there is a $(1+\varepsilon)$-isomorphism $f:B\to \mathbb{G}$ such that $\|f\circ g - id_A\|\leq \varepsilon$.
\end{definition}

In the sequel we will need the following result.

\begin{theorem}[{\cite[Theorem 4.1]{CDDK1}}]\label{thm:gurariiTypicalInP}
Let $\mathbb{G}$ be the Gurari\u{\i} space. The set $\isomtrclass[\mathcal I]{\mathbb{G}}$ is a dense $G_\delta$-set in $\mathcal{I}$ for any $\mathcal{I}\in\{\PP,\PP_\infty,\B\}$.
\end{theorem}

\section{Spaces with descriptively simple isometry and isomorphism classes}\label{sect:simpleClasses}

The topic of this section is to deal with spaces with descriptively simple isometry classes (see Subsection~\ref{subsec:smallIsometry}) and with spaces with descriptively simple isomorphism classes (see Subsection~\ref{subsect:smallIsomorphism}). The main outcome is Theorem~\ref{thm:Intro6} which follows from Theorem~\ref{thm:isometryClosedClass} and Theorem~\ref{thm:FSigmaIsomrfClasses}.

\subsection{Spaces with closed isometry classes}\label{subsec:smallIsometry}
In this subsection, we start our investigation of descriptive complexity of isometry classes with the main goal to prove the first part of Theorem~\ref{thm:Intro6}. Let us first observe that no isometry class can be open as every isometry class actually has an empty interior. Indeed, it follows from Proposition~\ref{prop:closureOfIsometryClasses} that the isometry class of every isometrically universal separable Banach space is dense. Since there are obviously many pairwise non-isometric universal Banach spaces we get that every open set (in all $\PP$, $\PP_\infty$ and $\B$) contains norms, resp. pseudonorms, defining distinct Banach spaces. The same argument can be also used to show that every isomorphism class has an empty interior.

\begin{lemma}\label{lem:Hilbertclosed}
$\isomtrclass{\ell_2}$ is closed in $\B$ and $\PP_\infty$.
\end{lemma}
\begin{proof}
Hilbert spaces are characterized among Banach spaces as those Banach spaces whose norm satisfies the parallelogram law, i.e. $\|x+y\|^2+\|x-y\|^2=2(\|x\|^2+\|y\|^2)$ for any pair of elements $x,y$. It is clear that a norm satisfies the parallelogram law if and only if it satisfies it on a dense set of vectors, therefore every norm, resp. pseudonorm, from $\B$, resp. $\PP_\infty$, satisfying the parallelogram law on $V$ defines a Hilbert space. Since norms, resp. pseudonorms, from $\B$, resp. $\PP_\infty$, define only infinite-dimensional spaces, they define spaces isometric to $\ell_2(\Nat)$. Since the parallelogram law is clearly a closed condition, we are done.
\end{proof}
\begin{remark}\label{rem:closedinP}
We note that here we need to work with the spaces $\B$ or $\PP_\infty$, since in $\PP$ the only space with closed isometry class is the trivial space. To show it, first notice that the trivial space is indeed closed. Next we show that any open neighborhood of a pseudonorm defining trivial space contains a pseudonorm defining arbitrary Banach space, which will finish our claim. Let such an open neighborhood be fixed. We may assume that it is of the form $\{\mu\in\PP\colon \mu(v_i)<\varepsilon, i\leq n\}$, where $v_1,\ldots,v_n\in V$ and $\varepsilon>0$. Let $m$ be such that all $v_i$, $i\leq n$, are in $\qeSpan\{e_j\colon j\leq m\}$. Let $X$ be an arbitrary separable Banach space and let $(f_i)_{i\in\Nat}\subseteq X$ be a sequence whose span is dense in $X$. We define $\mu\in\PP$ by $\mu(e_j)=0$, for $j\leq m$, and $\mu(\sum_{i\in I} \alpha_i e_{m+i})=\|\sum_{i\in I} \alpha_i f_i\|_X$, where $I\subseteq \Nat$ is finite and $(\alpha_i)_{i\in I}\subseteq \Rat$. This defines $\mu$ separately on $\qeSpan\{e_i\colon i\leq m\}$ and $\qeSpan\{e_i\colon i>m\}$, however the extension to the whole $V$ is unique. It is clear that $\mu$ is in the fixed open neighborhood and that $X_\mu\equiv X$.
\end{remark}
One may be interested whether there are other Banach spaces whose isometry class is closed. The answer is negative. This follows from another corollary of Proposition~\ref{prop:closureOfIsometryClasses} which we mentioned already in \cite[Corollary 2.11]{CDDK1}. Its proof is just an easy application of the Dvoretzky theorem.
\begin{lemma}[{\cite[Corollary 2.11]{CDDK1}}]\label{cor:Dvoretzky}
Let $X$ be a separable infinite-dimensional Banach space. Then $\isomtrclass[\B]{\ell_2}\subseteq \overline{\isomtrclass[\B]{X}}\cap\B$. The same holds if we replace $\B$ with $\PP_\infty$ or $\PP$.
\end{lemma}

The following theorem is now an immediate consequence of Lemmas~\ref{lem:Hilbertclosed} and~\ref{cor:Dvoretzky}.
\begin{theorem}\label{thm:isometryClosedClass}
$\ell_2$ is the only separable infinite-dimensional Banach space whose isometry class is closed in $\B$. The same holds if we replace $\B$ by $\PP_\infty$.
\end{theorem}

\subsection{$QSL_p$-spaces}
Before embarking on studying spaces with descriptively simple isomorphism classes, let us consider some natural closed subspaces of $\PP$, $\PP_\infty$ and $\B$. 

In \cite{Kwap}, Kwapie\' n denotes by $S_p$, resp. $SQ_p$, for $1\leq p<\infty$, the class of all Banach spaces isometric to a subspace of $L_p(\mu)$, resp. to a subspace of some quotient of $L_p(\mu)$, for some measure $\mu$. Note that a separable Banach space $X$ belongs to $S_p$, resp. $SQ_p$ if and only if it is isometric to a subspace of $L_p[0,1]$, resp. to a subspace of a quotient of $L_p[0,1]$, which easily follows from the fact that any separable $L_p(\mu)$ isometrically embeds into $L_p[0,1]$ as a complemented subspace (see e.g. Theorem~\ref{thm:scriptLpCharacterization} below).

Let us address the class $S_p$ first. We have the following simple lemma.
\begin{lemma}\label{lem:subspaceLp}
Let $1\leq p < \infty$. Put
\[M:=\{\mu\in\B\setsep X_\mu\text{ is isometric to a subspace of }L_p[0,1]\}.\]
Then $M$ is a closed set in $\B$ and we have
\[M= \overline{\isomtrclass[\B]{\ell_p}}\cap\B = \overline{\{\mu\in\B\setsep X_\mu\text{ is a $\mathcal{L}_{p,1+}$ space}\}}\cap\B.\]
The same holds if we replace $\B$ with $\PP_\infty$.
\end{lemma}
\begin{proof}
We recall the fact that a separable infinite-dimensional Banach space is isometric to a subspace of $L_p[0,1]$ if and only if it is finitely representable in $\ell_p$, see e.g.  \cite[Theorem 12.1.9]{albiacKniha}. The rest follows from Proposition~\ref{prop:closureOfIsometryClasses}. We refer the reader to Section~\ref{sect:Gdelta} for a definition of the class $\mathcal{L}_{p,1+}$.
\end{proof}

In the rest, we focus on the class $SQ_p$. Notice that for $p=1$, this class coincides with the class of all Banach spaces, and for $p=2$, this class consists of Hilbert spaces.

These Banach spaces are also called $QSL_p$-spaces in literature and since it seems this is the more recent terminology, this is what we will use further. It seems to be well known, see e.g. \cite{Runde}, that this class of spaces is characterized by Proposition~\ref{prop:Herz} below. This result was first essentially proved probably by Kwapien \cite{Kwap} (however, in his paper he considered the isomorphic variant only), for a more detailed explanation of the proof (and even for a generalization) one may consult e.g. the proof in \cite[Theorem 3.2]{Me96} which uses ideas from \cite{Pi90} and \cite{He83}. Let us note that, by Proposition~\ref{prop:Herz} and \cite[Proposition 0]{Herz}, the class of $QSL_p$-spaces coincides with the class of $p$-spaces considered already in 1971 by Herz \cite{Herz}.

\begin{proposition}\label{prop:Herz}
A Banach space $X$ is a $QSL_p$-space, if and only if for every real valued $(n,m)$-matrix $M$ satisfying $$\sum_{i=1}^n \left|\sum_{j=1}^m M(i,j) r_j \right|^p\leq \sum_{k=1}^m |r_k|^p,$$ for all $m$-tuples $r_1,\ldots,r_m\in\Rea$, we have $$\sum_{i=1}^n \left\|\sum_{j=1}^m M(i,j) x_j \right\|_X^p\leq \sum_{k=1}^m \|x_k\|_X^p,$$ for all $m$-tuples $x_1,\ldots,x_m\in X$.
\end{proposition}

Since it is clear that it suffices to verify the condition from Proposition~\ref{prop:Herz} only on dense tuples of vectors, and that this condition is closed, we immediately obtain the following.

\begin{proposition}
For every $1< p<\infty$, the set $$\{\mu\in \PP_\infty\colon X_\mu\text{ is a }QSL_p\text{-space}\}$$ is closed in $\PP_\infty$.

The same is true if $\PP_\infty$ is replaced by $\B$.
\end{proposition}

Denote now the set $\{\mu\in \PP_\infty\colon X_\mu\text{ is a }QSL_p\text{-space}\}$ by $QSL_p$. By Lemma~\ref{lem:subspaceLp}, for $1\leq p<\infty$, the set $M_p:=\{\mu\in\mathcal P_\infty\colon X_\mu\text{ is isometric to a subspace of }L_p[0,1]\}$ is closed. Clearly, $M_p\subseteq QSL_p$ (and for $p=2$ there is an equality).

If $p\neq 2$ then $M_p\neq QSL_p$ because there exists a separable infinite-dimensional Banach space which is isomorphic to a quotient of $L_p[0,1]$ but not to its subspace. Indeed, if $p=1$ this is easy since every separable Banach space is isomorphic to a quotient of $\ell_1$, see e.g. \cite[Theorem 2.3.1]{albiacKniha}. If $2<q<p<\infty$ then $\ell_q$ is isometric to a quotient of $L_p[0,1]$ (because its dual $\ell_{q'}$ embeds isometrically into $L_{p'}[0,1]$) but is not isomorphic to a subspace of $L_p[0,1]$, see e.g. \cite[Theorem 6.4.18]{albiacKniha}. Finally, if $1<p<2$ then by \cite[Corollary 2]{FKP77} there exists a subspace $X$ of $\ell_{p'}\subseteq L_{p'}[0,1]$ which is not isomorphic to a quotient of $L_{p'}[0,1]$\footnote[1]{More precisely, by \cite[Theorem 1]{FKP77} (see also e.g. \cite[Corollary 3.2]{FJ80}) for every $n\in\Nat$ there exists a subspace $E_n$ of $\ell_\infty^{2n}$ such that $gl(E_n)\geq K\sqrt{n}$ where $K>0$ is a constant independent of $n$ and $gl(E_n)$ is a quantity related to the notion of a ``$GL$-space'' (or a space with the ``Gordon-Lewis property''). This implies that if we denote by $E_n^{p'}$ the space $E_n$ endowed with the $\ell_{p'}$-norm, we obtain $gl(E_n^{p'})\geq K\sqrt{n}d_{BM}(\ell_\infty^{2n},\ell_{p'}^{2n})^{-1} = K2^{-1/{p'}}n^{1/2-1/{p'}}\to \infty$; hence, $X:=(\bigoplus E_n^{p'})_{p'}$, the $\ell_{p'}$-sum of the spaces $E_n^{p'}$, is isometric to a subspace of $\ell_{p'}$ but it is not a $GL$-space. If $X$ was isomorphic to a quotient of $L_{p'}[0,1]$ then $X^*$ would be isomorphic to a subspace of $L_{p}[0,1]$ which would imply that $X^*$ and $X$ are $GL$-spaces (see e.g. \cite[Proposition 17.9 and Proposition 17.10]{diestelKniha}), a contradiction.} and so $X^*$ is isometric to a quotient of $L_p[0,1]$ which is not isomorphic to a subspace of $L_p[0,1]$. We would like to thank Bill Johnson for providing us these examples.

Moreover, we have the following.
\begin{proposition}
For $p\in \left[1,2\right)\cup \left(2,\infty\right)$, the set $M_p$ has an empty interior in $QSL_p$.
\end{proposition}
\begin{proof}
Fix $p\in \left[1,2\right)\cup \left(2,\infty\right)$. Pick  $\mu\in QSL_p$ such that $X_\mu$ does not isometrically embed as a subspace into $L_p[0,1]$ (such a space exists, see the examples above). Let $U$ be now a basic open neighborhood of some $\nu\in QSL_p$. Since the class of $QSL_p$-spaces is clearly closed under taking $\ell_p$-sums (see e.g. \cite{Runde}), $X_\nu\oplus_p X_\mu$ is still a $QSL_p$-space. It is easy to define $\nu'\in U$ so that $X_{\nu'}$ is isometric to $X_\nu\oplus_p X_\mu$. Now since $X_\mu$ does not isometrically embed as a subspace into $L_p[0,1]$, neither $X_{\nu'}$ does. By \cite[Theorem 12.1.9]{albiacKniha}, $X_{\nu'}$ is not finitely representable in $\ell_p$, so also not in $L_p[0,1]$ (by \cite[Proposition 12.1.8]{albiacKniha}). It follows from Proposition~\ref{prop:closureOfIsometryClasses} that there exists a basic open neighborhood $U'$ of $\nu'$ avoiding $M_p$. Now $U\cap U'$ is a non-empty open subsets of $U$ avoiding $M_p$ and we are done.
\end{proof}
\begin{corollary}\label{cor:LpnotQSLpgeneric}
For $p\in \left[1,2\right)\cup \left(2,\infty\right)$, $L_p[0,1]$ is not a generic $QSL_p$-space.
\end{corollary}

\subsection{Spaces with descriptively simple isomorphism classes}\label{subsect:smallIsomorphism}
The main result of this subsection, and one of the main results of the whole paper, is the second part of Theorem~\ref{thm:Intro6}. That is, we prove the following.
\begin{theorem}\label{thm:FSigmaIsomrfClasses}
The Hilbert space $\ell_2$ is characterized as the unique, up to isomorphism, separable infinite-dimensional Banach space $X$ such that $\isomrfclass{X}$ is an $F_\sigma$-set in $\B$. The same holds if we replace $\B$ with $\PP_\infty$.

Moreover, $\isomrfclass{\ell_2}$ is $F_\sigma$-complete.
\end{theorem}
While there are several Banach spaces whose isometry classes have low complexity, as we shall see also in the next sections, there are reasons to suspect that isomorphism classes are rather complicated in general. Theorem~\ref{thm:FSigmaIsomrfClasses} is a strong evidence that $\ell_2$ is quite unique with respect to its property of having a simple isomorphism class. Another piece of evidence which we state and prove before proving Theorem~\ref{thm:FSigmaIsomrfClasses} is the following result.

\begin{proposition}\label{prop:classOfl2}
No isomorphism class can be closed in $\PP_\infty$, $\B$ and $\PP$, and with the possible exception of spaces isomorphic to $\mathbb{G}$ for which we do not know the answer, no isomorphism class can even be a $G_\delta$-set.

\end{proposition}
\begin{proof}
Let $X$ be a separable infinite-dimensional Banach space. We show that $\isomrfclass{X}$ is dense (we show the argument only for $\PP_\infty$, the other cases are analogous). Let $F$ be a finite-dimensional Banach space. It is well known that every finite-dimensional space is complemented in any infinite-dimensional Banach space, so we have $X\simeq F\oplus_1 Y$ for some Banach space $Y$. Since $F$ was arbitrary, it follows that every separable Banach space is finitely representable in $\isomrfclass{X}$, so by Proposition~\ref{prop:closureOfIsometryClasses}, $\overline{\isomrfclass{X}}=\PP_\infty$, hence $\isomrfclass{X}$ is dense.

It follows that $\isomrfclass{X}$ cannot be closed for any $X$ because it is dense and there are obviously two non-isomorphic spaces.
Moreover, if $X$ is not isomorphic to the Gurari\u{\i} space then $\isomrfclass{X}$ cannot be a $G_\delta$-set since by Theorem~\ref{thm:gurariiTypicalInP}, the isometry class of the Gurari\u{\i} space is a dense $G_\delta$-set, so it would have non-empty intersection with $\isomrfclass{X}$ otherwise.
\end{proof}

 Besides $\ell_2$, whose isomorphism clas is $F_\sigma$, it is proved in \cite[Theorem 4.12]{GS18} that separable Banach spaces determined by their pavings have $\boldsymbol{\Sigma}_4^0$ isomorphism classes in any admissible topology (see Remark~\ref{rem:admissible} where we recall some basics about admissible topologies). We refer the interested reader to the text before Theorem~\ref{thm:paving} for a definition of spaces determined by their pavings. Here we just briefly note that this class of spaces was introduced by Johnson, Lindenstrauss and Schechtman in \cite{JLS96} and that there are known examples of separable Banach spaces determined by their pavings not isomorphic to $\ell_2$ (e.g. certain $\ell_2$-sums of finite-dimensional spaces are such). The second main result of this section is the following improvement of the estimate mentioned above.
 \begin{theorem}\label{thm:pavingDetermined2}
Let $X$ be a separable infinite-dimensional Banach space that is determined by its pavings. Then $\isomrfclass{X}$ is a $G_{\delta\sigma}$-set in $\PP_\infty$. In particular, it is a $G_{\delta\sigma}$-set in $\PP$ and in any admissible topology.
\end{theorem}
Let us start with the proof of Theorem~\ref{thm:FSigmaIsomrfClasses}.
\begin{proof}[Proof of Theorem~\ref{thm:FSigmaIsomrfClasses}]
First we show that isomorphism class is $F_\sigma$. This was proved for an admissible topology on $SB_\infty$ in \cite[Theorem 4.3]{GS18}. The same proof, which we briefly sketch, works also for $\PP_\infty$ and $\B$. By Kwapie\'{n}'s theorem (see e.g. \cite[Theorem 7.4.1]{albiacKniha}) a separable infinite-dimensional Banach space is isomorphic to $\ell_2$ if and only if it is of type $2$ and of cotype $2$. It is clear from the definition of type and cotype (see e.g. \cite[Definition 6.2.10]{albiacKniha}) that these properties are $F_\sigma$-conditions. So to show that $\isomrfclass{\ell_2}$ is even $F_\sigma$-complete, by Lemma~\ref{lem:F_sigma-hard} it suffices to show that $\isomrfclass{\ell_2}$ is not a $G_\delta$-set, which we have already proved.
An alternative proof showing that the isomorphism class $\isomrfclass{\ell_2}$ is an $F_\sigma$-set follows from \cite[Theorem 2']{Kwap} (see also Remark 4 therein) which provides a formula defining spaces isomorphic to $\ell_2$ and which obviously defines an $F_\sigma$-set (in $\PP_\infty$ and $\B$).\medskip

Next we show that if a separable infinite-dimensional Banach space $X$ is not isomorphic to $\ell_2$, then $\isomrfclass{X}$ is not $F_\sigma$ in $\B$. In what follows, we denote by $\T$ the set of finite tuples (including empty) of natural numbers without repetition. The length of $\gamma\in \T$ is denoted by $|\gamma|$ and its range by $\operatorname{rng}(\gamma)$. Moreover, for every $\gamma\in \T$ and every $\mu\in\B$ we put
\[
\mathbb{M}^{\gamma}_\mu:=\Big\{\nu\in\B\colon \textnormal{for every $(a_i)_{i=1}^{|\gamma|}\in\Rat^{|\gamma|}$ we have } \nu\Big(\sum_{i=1}^{|\gamma|} a_i e_i\Big) = \mu\Big(\sum_{i=1}^{|\gamma|} a_i e_{\gamma(i)}\Big)\Big\}.
\]

In order to get a contradiction assume that $(F_n)_{n=1}^\infty$ are closed sets in $\B$ such that $\isomrfclass{X} = \bigcup_{n=1}^\infty F_n$.
\begin{claim*}
For every $\mu\in\B$ with $\isomtrclass{X_\mu}\subseteq \bigcup_{n=1}^\infty F_n$ there exist $\gamma\in \T$ and $m\in\Nat$ such that we have $\mathbb{M}^{\gamma'}_\mu\cap F_m\neq\emptyset$ for every $\gamma'\in \T$ with $\gamma'\supseteq\gamma$.
\end{claim*}
\begin{proof}[Proof of the claim]
Suppose the statement is not true. In particular, it does not hold for $\gamma = \emptyset$ and $m=1$. That is, there is some $\gamma'_1\in \T$ so that $\mathbb{M}^{\gamma_1'}_\mu \cap F_1=\emptyset$. If $1\in\operatorname{rng}(\gamma'_1)$, we set $\gamma_1=\gamma'_1$. Otherwise, we set $\gamma_1={\gamma'_1}^{\frown} (1)$.

In the next step, we use that the statement is not true for $\gamma_1$ and $m=2$ to obtain $\gamma'_2\in \T$, $\gamma'_2\supseteq \gamma_1$ so that $\mathbb{M}^{\gamma'_2}_\mu \cap F_2=\emptyset$. If $2\in\operatorname{rng}(\gamma'_2)$, we set $\gamma_2=\gamma'_2$. Otherwise, we set $\gamma_2={\gamma'_2}^\frown{} (2)$.

We continue analogously. At the end of the recursion, we obtain a bijection $\pi:\Nat\to\Nat$ such that $\pi\supseteq \gamma_n$ for every $n\in\Nat$ and $\mathbb{M}^{\gamma_n}_\mu\cap F_n=\emptyset$ for every $n\in\Nat$. Consider $\mu_0\in\B$ given as 
\[
\mu_0\Big(\sum_{i=1}^k a_i e_i\Big):=\mu\Big(\sum_{i=1}^k a_i e_{\pi(i)}\Big),\quad k\in\Nat,\; (a_i)_{i=1}^k\in \Rat^k.
\]
Then the linear mapping given by $e_i\mapsto e_{\pi(i)}$, $i\in\Nat$, witnesses that $X_{\mu_0}\equiv X_\mu$ and $\mu_0\in \mathbb{M}^{\gamma_n}_\mu$ for every $n\in\Nat$. Thus, $\mu_0\notin \bigcup_{n=1}^\infty F_n$ which is in contradiction with $\mu_0\in\isomtrclass{X_\mu}\subseteq  \bigcup_{n=1}^\infty F_n$.
\end{proof}
Since $X\not\simeq \ell_2$, by the celebrated solution to the homogeneous subspace problem following from the results of Komorowski and Tomczak-Jaegermann (\cite{KoTJ95}) and of Gowers (\cite{G02}), it must contain an infinite-dimensional closed subspace $Y\subseteq X$ that is not isomorphic to $X$. Let $I\subseteq \Nat$ be an infinite subset and $\{x_n\}_{n\in\Nat}$ a sequence of linearly independent vectors in $X$ so that

\begin{enumerate}
    \item $\overline{\Span}\{x_n\}_{n\in\Nat}=X$, $\overline{\Span}\{x_n\colon n\in I\}=Y$, and
    \item\label{it:sums} for every finite set $F\subseteq \Nat$ we have \[\overline{\Span}\{x_n\}_{n\in I \cup F} = \overline{\Span}\{x_n\}_{n\in F\cap I}\oplus \overline{\Span}\{x_n\}_{n\in F\setminus I}\oplus \overline{\Span}\{x_n\}_{n\in I \setminus F}.\]
\end{enumerate}
Such a choice is possible e.g. by finding a Markushevich basis $\{x_n\}_{n\in I}$ on $Y$ (see \cite[Theorem 1.22]{HMVZ}), extending it to a Markushevich basis $\{x_n\}_{n\in \Nat}$ on $X$ (see \cite[Theorem 1.45]{HMVZ}) and then observing that for any $J\subseteq \Nat$ we have that $\{x_n\}_{n\in J}$ is a fundamental biorthogonal system of $\overline{\Span}\{x_n\}_{n\in J}$ and so \eqref{it:sums} holds (see \cite[Fact 1.5]{HMVZ}).

We define $\mu\in\B$ as
\[
\mu\Big(\sum_{i=1}^k a_i e_i\Big):=\Big\|\sum_{i=1}^k a_i x_i\Big\|_X,\quad k\in\Nat,\;(a_i)_{i=1}^k\in\Rat^k.
\]
Then $\isomtrclass{X_\mu}=\isomtrclass{X}\subseteq \bigcup_{n=1}^\infty F_n$ and so, by the claim above, there exist $\gamma\in \T$ and $m\in\Nat$ with $\mathbb{M}^{\gamma'}_\mu\cap F_m\neq\emptyset$ for every $\gamma'\in \T$ with $\gamma'\supseteq \gamma$. Consider now the space $Z:=(\overline{\Span}\{e_i\colon i\in I\cup \operatorname{rng}(\gamma)\},\mu)\subseteq X_\mu$. Fix some $\widetilde I\subseteq I$ such that $|I\setminus\widetilde I|=|\operatorname{rng}(\gamma)\setminus I|$ and such that $(I\setminus\widetilde I)\cap\operatorname{rng}(\gamma)=\emptyset$. Define $\widetilde Z:=(\overline{\Span}\{e_i\colon i\in \widetilde I\cup \operatorname{rng}(\gamma)\},\mu)$. Then, using \eqref{it:sums}, we have 
\[
\widetilde{Z} \simeq \overline{\Span}\{x_n\}_{n\in \widetilde{I}}\oplus \Rea^{| \operatorname{rng}(\gamma)\setminus\widetilde{I}|}\simeq \overline{\Span}\{x_n\}_{n\in \widetilde{I}}\oplus \Rea^{| I\setminus\widetilde{I}|}\simeq Y,
\]
and so $\widetilde Z\not\simeq X$. Let $\varphi:\Nat\to \operatorname{rng}(\gamma)\cup \widetilde I$ be a bijection with $\varphi\supseteq\gamma$. We define $\nu\in\B$ by
\[
\nu\Big(\sum_{i=1}^k a_i e_i\Big):=\mu\Big(\sum_{i=1}^k a_i e_{\varphi(i)}\Big),\quad k\in\Nat,\;(a_i)_{i=1}^k\in\Rat^k.
\]
Clearly, $X_\nu\equiv \widetilde Z\not\simeq X$.

We \emph{claim} that $\nu\in F_m$. This will be in contradiction with the fact that $F_m\subseteq \isomrfclass{X}$. Since $F_m$ is closed, it suffices to check that each basic open neighborhood of $\nu$ intersects $F_m$. Pick $v_1,\ldots, v_l\in V$ and $\varepsilon>0$. We need to find $\mu'\in F_m$ so that $|\mu'(v_j)-\nu(v_j)|<\varepsilon$ for every $j\leq l$. 

Let $L\in\Nat$, $L\geq |\gamma|$, be such that $v_1,\ldots,v_l\in\Span\{e_i\colon i\leq L\}$. Since $\varphi|_{\{1,\ldots,L\}}\supseteq \gamma$, we may pick $\mu'\in \mathbb{M}^{\varphi|_{\{1,\ldots,L\}}}_\mu\cap F_m$. Then
\[
\mu'\Big(\sum_{i=1}^L a_i e_i\Big) = \mu\Big(\sum_{i=1}^L a_i e_{\varphi(i)}\Big) = \nu\Big(\sum_{i=1}^L a_i e_i\Big),\quad (a_i)_{i=1}^L\in\Rat^L.
\]
In particular, $\mu'(v_j)=\nu(v_j)$, $j\leq l$, as desired.
\end{proof}
In the remainder of this section we head towards the proof of Theorem~\ref{thm:pavingDetermined2}.
  Following \cite{JLS96}, we say that an increasing sequence $E_1\subseteq E_2\subseteq\ldots$ of finite-dimensional subspaces of a separable Banach space $X$ whose union is dense is a \emph{paving of $X$}. A separable Banach space $X$ is \emph{determined by its pavings} if whenever $Y$ is a Banach space for which there are pavings $\{E_n\}_{n=1}^\infty$ of $X$ and $\{F_n\}_{n=1}^\infty$ of $Y$ with $\sup_{n\in\Nat}d_{BM}(E_n,F_n)<\infty$, then $Y$ is isomorphic to $X$. We refer the reader to \cite{JLS96} for details and examples.

We start with a theorem that is interesting on its own.

\begin{theorem}\label{thm:paving}
Let $X$ be a separable infinite-dimensional Banach space, $\{E_n\}_{n=1}^\infty$ a paving of $X$ and $\lambda\geq 1$. Then the set
\[\begin{split}
\mathcal{Z}:=\big\{\mu\in\PP\colon & \textnormal{for every $\varepsilon>0$ there is a paving $\{F_k\}_{k=1}^\infty$ of $X_\mu$ and an increasing } \\
& \textnormal{sequence $(n_k)_{k=1}^\infty\in \Nat^\Nat$ with $\sup _{k\in\Nat}d_{BM}(F_k,E_{n_k})\leq \lambda+\varepsilon$}\big\}
\end{split}\]
is a $G_\delta$-set in $\PP$.
\end{theorem}

In the proof we will need the following observation, which is well-known and easy to prove. We use the formulation from \cite[Lemma 4.3]{CDDK1}.

\begin{lemma}\label{lem:approx2}
Given a basis $\b_E = \{e_1,\ldots,e_n\}$ of a finite-dimensional Banach space $E$, there is $C>0$ and a function $\phi_2^{\b_E}:[0,C)\to[0,\infty)$ continuous at zero with $\phi_2^{\b_E}(0)=0$ such that whenever $X$ is a Banach space with $E\subseteq X$ and $\{x_i\setsep i\leq n\}\subseteq X$ are such that $\|x_i-e_i\|<\varepsilon$, $i\leq n$, for some $\varepsilon<C$, then the linear operator $T:E\to X$ given by $T(e_i):=x_i$ is a $(1+\phi_2^{\b_E}(\varepsilon))$-isomorphism and $\|T-Id_E\|\leq \phi_2^{\b_E}(\varepsilon)$.
\end{lemma}

\begin{proof}[Proof of Theorem~\ref{thm:paving}]
For each $E_n$, $n\in\Nat$, we fix its basis (which will be used only in order to know what $Z\approximates[E_n]{K}$ means for a finite-dimensional space $Z$ and $K>1$). For each finite tuple $\vec v$ of elements from $V$ we set $S_{\vec v}$ to be the set of all finite tuples $\vec w$, $\Rat$-linearly independent in $V$, such that each element of $\vec v$ (considered as an element of $c_{00}$) lies in $\Span\vec w$. For $m\in\Nat$ and a finite tuple $\vec v$ of elements from $V$ we set
\[\begin{split}
P(\vec v,m):=\big\{\mu\in\PP\colon & \textnormal{if $\mu$ restricted to }\Span\{v_1,\ldots,v_{|\vec v|}\}\subseteq c_{00}\textnormal{ is a norm, then }\\
& \textnormal{there exist $\vec w\in S_{\vec v}$ and $ n\in\Nat$ such that }\\
& \textnormal{$\mu$ restricted to }\Span\{w_1,\ldots,w_{|\vec w|}\}\subseteq c_{00}\textnormal{ is a norm and }\\
& (((\Span\{w_1,\ldots,w_{|\vec w|}\},\mu),\vec w)\approximates[E_{n}]{\sqrt{\lambda + \frac{1}{m}}})\big\}.
\end{split}\]
Then, using the observation that $\{\mu\in\PP\colon \mu\textnormal{ restricted to }\Span\{v_1,\ldots,v_{|\vec v|}\}\subseteq c_{00}\textnormal{ is a norm}\}$ is open due to Lemma~\ref{lem:infiniteDimIsGDelta}, $P(\vec v,m)$ is the union of a closed and an open set, so it is a $G_\delta$-set.

Denote by $LI$ the set of all finite tuples $\vec v = (v_1,\ldots,v_{|\vec v|})$ of elements from $V$ which are linearly independent in $c_{00}$. We now set
\[
\mathcal{G}:=\bigcap_{m\in\Nat} \bigcap_{\vec v\in LI} P(\vec v,m),
\]
which is clearly a $G_{\delta}$-set. We shall prove that $\mathcal{G}=\mathcal{Z}$.

If $\mu\in\mathcal{G}$, it is clear that for every $m$, we can recursively build an increasing sequence $\{F_k\}_{k=1}^\infty$ of finite-dimensional subspaces whose union is dense in $X_\mu$ such that we have $d_{BM}(F_k,E_{n_k})\leq \lambda + \frac{1}{m}$ for every $k\in\Nat$. It follows that $\mu\in \mathcal{Z}$.

On the other hand, pick $\mu\in\mathcal{Z}$. In what follows for $x\in c_{00}$ we denote by $[x]\in X_\mu$ the equivalence class corresponding to $x$. Pick some $m\in\Nat$ and an $n$-tuple $\vec v\in LI$ such that $\mu$ restricted to $\Span\{v_1,\ldots,v_n\}\subseteq c_{00}$ is a norm, so $\{v_1,\ldots,v_n\}$ is a basis of $\Span\{v_1,\ldots,v_n\}$. Pick $\lambda'\in(\lambda,\lambda + \tfrac{1}{m})$ and $\delta>1$ with $\delta\lambda'<\lambda + \frac{1}{m}$. Since $\mu\in\mathcal{Z}$, there is an increasing sequence $\{F_k\}_{k=1}^\infty$ of finite-dimensional subspaces whose union is dense in $X_\mu$ such that $\sup_{k\in\Nat}d_{BM}(F_k,E_{n_k})\leq \lambda'$. By \cite[Section 17, Theorem 6]{LaceyBook}, we can find a finite dimensional subspace $\Span\{[v_1],\ldots,[v_n]\}\subseteq Y\subseteq X_\mu$ and $k\in\Nat$ such that $d_{BM}(Y,F_k)\leq \delta$ so $d_{BM}(Y,E_{n_k})\leq \delta\lambda'$. Select $y_{n+1},\ldots, y_{\dim Y}\in Y$ such that $\mathfrak{b} = \{[v_1],\ldots, [v_n], y_{n+1},\ldots, y_{\dim Y}\}$ is a basis of $Y$. Let $\phi_2^{\mathfrak{b}}$ be the function from Lemma~\ref{lem:approx2} and let $\eta>0$ be such that $\delta\lambda'(1+\phi_2^{\mathfrak{b}}(\eta))^2<\lambda+\frac{1}{m}$. Further, for every $n+1\leq i\leq \dim Y$ pick $v_i\in V$ with $\|[v_i]-y_i\|_{X_\mu}<\eta$.
Then $(\mu,[v_1],\ldots,[v_{\dim Y}])\approximates[Y]{1+\phi_2^{\mathfrak{b}}(\eta)}$ so $\mu$ restricted to $\Span\{v_1,\ldots,v_{\dim Y}\}\subseteq c_{00}$ is a norm and $\Span\{[v_1],\ldots,[v_{\dim Y}]\}\subseteq X_\mu$ is isometric to $(\Span\{v_1,\ldots,v_{\dim Y}\},\mu)$. Since $d_{BM}((\Span\{v_1,\ldots,v_{\dim Y}\},\mu),E_{n_k})<\delta\lambda'\cdot(1+\phi_2^{\mathfrak{b}}(\eta))^2<\lambda + \frac{1}{m}$, there exists a surjective isomorphism $T:E_{n_k}\to (\Span\{v_1,\ldots,v_{\dim Y}\},\mu)$  with $\max\{\|T\|,\|T^{-1}\|\}<\sqrt{\lambda + \frac{1}{m}}$. By Lemma~\ref{lem:approx2}, we may without loss of generality assume that $w_i:=T(e_i)\in V$ for every $i\leq \dim Y$. Then $\mu$ restricted to $\Span\{w_1,\ldots,w_{\dim Y}\}\subseteq c_{00}$ is a norm, $\vec w\in S_{\vec v}$ and $((\mu,\vec w)\approximates[E_{n_k} ]{\sqrt{\lambda + \frac{1}{m}}})$.
\end{proof}

\begin{proof}[Proof of Theorem~\ref{thm:pavingDetermined2}]
Pick a paving $\{E_n\}_{n=1}^\infty$ of $X$. It is easy to see that for every $\mu\in\PP_\infty$ the Banach space $X_\mu$ is isomorphic to $X$ if and only if $\mu$ belongs to the set $\mathcal{Z}$ from Theorem~\ref{thm:paving} for some $\lambda\geq 1$. The ``In particular'' part follows since $\PP_\infty$ is a $G_\delta$-set in $\PP$. For admissible topologies, the result follows by applying \cite[Theorem 3.3]{CDDK1}.
\end{proof}

\section{Spaces with \texorpdfstring{$G_\delta$}{Gdelta} isometry classes}\label{sect:Gdelta}
In this section, we investigate Banach spaces whose isometry classes are $G_\delta$-sets, or even $G_\delta$-complete sets. The main results here are Theorems~\ref{thm:gurariiTypicalInPGeneralized}, \ref{thm:classOfLp} and \ref{thm:scriptLp} which imply the first and the last part of Theorem~\ref{thm:Intro2}.

Besides $\ell_2$, whose isometry class is actually closed, we have already mentioned in Theorem~\ref{thm:gurariiTypicalInP} that the isometry class of the Gurari\u{\i} space is a $G_\delta$-set in $\PP_\infty$ and $\B$. We start the section with some basic corollaries of that result; in particular, that the isometry class of $\mathbb{G}$ is even a $G_\delta$-complete set.
\medskip

Since for any separable infinite-dimensional Banach space $X$ we obviously have $\isomtrclass[\B]{X} = \isomtrclass[\PP_\infty]{X}\cap \B$, it is sufficient to formulate  our positive results in the coding of $\PP_\infty$ and negative results in the coding of $\B$.

\begin{lemma}\label{lem:uniqueGDelta}
Let $X$, $Y$ be separable infinite-dimensional Banach spaces such that $X$ is finitely representable in $Y$ and $Y$ is finitely representable in $X$. If $\isomtrclass{X}$ is a $G_\delta$-set in $\B$ and $X\not\equiv Y$, then 
\begin{enumerate}[(i)]
    \item $\isomtrclass{Y}$ is not a $G_\delta$-set in $\B$.
    \item $\isomtrclass{X}$ is a $G_\delta$-complete set in $\B$.
\end{enumerate}
\end{lemma}
\begin{proof}
Recall that by Proposition~\ref{prop:closureOfIsometryClasses} we have that both $\isomtrclass{X}$ and $\isomtrclass{Y}$ are dense in
\[
N:=\{\nu\in\B\setsep X_\nu\text{ is finitely representable in }X\}.
\]
(i): If both $\isomtrclass{X}$ and $\isomtrclass{Y}$ are $G_\delta$-sets, by the Baire theorem we have that $\isomtrclass{X}\cap \isomtrclass{Y}$ is comeager in $N$. Thus, the intersection cannot be an empty set and we obtain $X\equiv Y$.\\
(ii): Since $X\not\equiv Y$, we have that $\isomtrclass{X}$ has empty interior in $N$. But it is also comeager in $N$, and so it cannot be $F_\sigma$. Therefore it is a $G_\delta$-complete set by Lemma~\ref{lem:F_sigma-hard}.
\end{proof}

\begin{corollary}\label{cor:gurariiComplexity}
$\mathbb{G}$ is the only isometrically universal separable Banach space whose isometry class is a $G_\delta$-set in $\B$. The same holds if we replace $\B$ by $\PP_\infty$.

Moreover, $\isomtrclass{\mathbb{G}}$ is a $G_\delta$-complete set in both $\PP_\infty$ and $\B$.
\end{corollary}
\begin{proof}
By Theorem~\ref{thm:gurariiTypicalInP}, the isometry class of $\mathbb{G}$ is a $G_\delta$-set. Let $X$ be an isometrically universal separable Banach space. By Lemma~\ref{lem:uniqueGDelta}, if $X\not\equiv \mathbb{G}$ then $\isomtrclass[\B]{X}$ is not a $G_\delta$-set in $\B$ (and so neither in $\PP_\infty$).

For the ``moreover'' part we use Lemma~\ref{lem:uniqueGDelta} and any Banach space $X$ not isometric to $\mathbb{G}$ that is finitely representable in $\mathbb{G}$ and vice versa (e.g. any other universal separable Banach space or $c_0$).
\end{proof}

The same proof gives us actually the following strengthening. Let us recall that by Maurey–Pisier theorem, see \cite{MP76} or \cite[Theorem 12.3.14]{albiacKniha}, a Banach space $X$ has no nontrivial cotype if and only if $\ell_\infty$ is finitely-representable in $X$ (and yet equivalently, $c_0$ is finitely-representable in $X$).
\begin{theorem}\label{thm:gurariiTypicalInPGeneralized}
$\mathbb{G}$ is the only separable Banach space with no nontrivial cotype whose isometry class is a $G_\delta$-set in $\B$. The same holds if we replace $\B$ by $\PP_\infty$.
\end{theorem}
\begin{proof}
Any separable Banach space is finitely representable in $c_0$, so by Lemma~\ref{lem:uniqueGDelta} there is at most one Banach space $X$ such that $c_0$ is finitely representable in $X$ and $\isomtrclass{X}$ is a $G_\delta$-set. By Theorem~\ref{thm:gurariiTypicalInP}, $\isomtrclass{\mathbb{G}}$ is a $G_\delta$-set.
\end{proof}

\subsection{$L_p$-spaces}
Let us recall that a Banach space $X$ is said to be an $\mathcal{L}_{p,\lambda}$-space (with $1\leq p\leq\infty$ and $\lambda\geq 1$) if every finite-dimensional subspace of $X$ is contained in another finite-dimensional
subspace of $X$ whose Banach-Mazur distance $d_{BM}$ to the corresponding $\ell_p^n$ is at most $\lambda$.
A space $X$ is said to be an $\mathcal{L}_{p}$-space, resp. $\mathcal{L}_{p,\lambda+}$-space, if it is an $\mathcal{L}_{p,\lambda'}$-space for some for some $\lambda'\geq 1$, resp. for every $\lambda'>\lambda$.

The main result of this subsection is the following.

\begin{theorem}\label{thm:classOfLp}
For every $1\leq p<\infty$, $p\neq 2$, the isometry class of $L_p[0,1]$ is a $G_\delta$-complete set in $\B$ and $\PP_\infty$.

Moreover, $L_p[0,1]$ is the only separable $\mathcal{L}_{p,1+}$-space whose isometry class is a $G_\delta$-set in $\B$, and the same holds if we replace $\B$ by $\PP_\infty$.
\end{theorem}

\begin{remark}\label{rem:scriptLpAgain}
It is easy to see (e.g. using \cite[Section 17, Theorem 6]{LaceyBook}) that for every $p\in[1,\infty]$ and $\lambda\geq 1$ we have that a separable infinite-dimensional Banach space $Y$ is an $\mathcal{L}_{p,\lambda+}$-space if and only if for every $\varepsilon>0$ there is an increasing sequence $\{F_k\}_{k=1}^\infty$ of finite-dimensional subspaces whose union is dense in $Y$ such that $d_{BM}(\ell_p^{\dim F_k},F_k)\leq \lambda+\varepsilon$ for every $k\in\Nat$.
\end{remark}

The next theorem is a crucial step in proving Theorem~\ref{thm:classOfLp}. However, it is also of independent interest and its corollary improves the related result from \cite{GS18}.

\begin{theorem}\label{thm:scriptLp}
Let $1\leq p\leq \infty$ and $\lambda\geq 1$. The class of separable infinite-dimensional $\mathcal{L}_{p,\lambda+}$-spaces is a $G_\delta$-set in $\PP$. In particular, the class of separable infinite-dimensional $\mathcal{L}_{p,\lambda+}$-spaces is a $G_\delta$-set in $\PP_\infty$.
\end{theorem}
\begin{proof}
This is an immediate consequence of Theorem~\ref{thm:paving}, Remark~\ref{rem:scriptLpAgain}.
\end{proof}
Note that for $1\leq p\leq \infty$ the class of $\mathcal{L}_p$-spaces is obtained as the union $\bigcup_{\lambda\geq 1} \mathcal{L}_{p,\lambda+}$. It is shown in \cite[Proposition 4.5]{GS18} that the class of separable $\mathcal{L}_p$-spaces is a $\boldsymbol{\Sigma}_4^0$-set in an admissible topology.  It is immediate from Theorem~\ref{thm:scriptLp} (and using \cite[Theorem 3.3]{CDDK1}) that we have a better estimate.
\begin{corollary}\label{cor:scriptLp}
For every $1\leq p\leq \infty$ the class of separable $\mathcal{L}_p$-spaces is a $G_{\delta \sigma}$-set in $\PP$ and any admissible topology.
\end{corollary}
\begin{proof}
Note that any finite-dimensional space is an $\mathcal{L}_p$-space, so the class of separable $\mathcal{L}_p$-spaces may be written as the union of $\PP\setminus\PP_\infty$ (which is an $F_\sigma$-set by \cite[Corollary 2.5]{CDDK1}) and $\{\mu\in\PP_\infty\colon X_\mu\text{ is an $\mathcal{L}_p$-space}\}$ (which is $G_{\delta \sigma}$-set by Theorem~\ref{thm:scriptLp} above).
\end{proof}
Let us recall the following classical result.
\begin{theorem}[Lindenstrauss, Pe\l czy\'{n}ski]\label{thm:scriptLpCharacterization}
For every $1\leq p<\infty$ and a separable infinite-dimensional Banach space $X$ the following assertions are equivalent.
\begin{itemize}
    \item $X$ is an $\mathcal{L}_{p,1+}$-space.
    \item $X$ is isometric to a separable $L_p(\mu)$ space for some measure $\mu$.
    \item $X$ is isometric to one of the following spaces
\[
L_p[0,1],\quad L_p[0,1]\oplus_p \ell_p,\quad \ell_p,\quad L_p[0,1]\oplus_p \ell_p^n\; \text{ (for some }n\in\Nat).
\]
\end{itemize}
\end{theorem}
\begin{proof}
By \cite[Section 7, Corollaries 4 and 5]{liPe68}, a separable Banach space is an $\mathcal{L}_{p,1+}$-space if and only if it is isometric to an $L_p(\mu)$ space for some measure $\mu$. Finally, note that every separable infinite-dimensional $L_p(\mu)$ space is isometric to one of the spaces mentioned above, see e.g. \cite[p. 137-138]{albiacKniha}.
\end{proof}

Recall that given a finite sequence $(z_n)_{n\in N}$ in a Banach space $Z$, the symbol $(z_n)\approximates[\ell_p^N]{K}$ means that $K^{-1}\left(\sum_{i\in N}|a_i|^p\right)^{1/p} < \|\sum_{i\in N} a_i z_i\| < K\left(\sum_{i\in N}|a_i|^p\right)^{1/p}$ for every $a\in c_{00}^N$. If $(z_n)$ is isometrically equivalent to the $\ell_p^N$ basis (that is, $(z_n)\approximates[\ell_p^N]{1+\varepsilon}$ for every $\varepsilon > 0$), we write $(z_n)\equiv \ell_p^N$.

\begin{theorem}\label{thm:characterizationLp}
Let $1\leq p<\infty$, $p\neq 2$, and let $X$ be a separable infinite-dimensional $\mathcal{L}_{p,1+}$ space. Then the following assertions are equivalent.
\begin{enumerate}[(i)]
    \item $X$ is isometric to $L_p[0,1]$.
    \item For every $x\in S_X$ the following condition is satisfied
    \[\forall N\in\Nat\;\exists x_1,\ldots,x_N\in X:\quad 
      (x_i)_{i=1}^N \equiv \ell_p^N \text{ and }N^{1/p}\cdot x = \sum_{i=1}^N x_i.\]
     \item For every $x\in S_X$ the following condition is satisfied
     \[\forall \varepsilon > 0\;\exists x_1,x_2\in X:\quad 
     (x_1,x_2)\approximates[\ell_p^2]{1+\varepsilon}  \text{ and }2^{1/p}\cdot x = x_1 + x_2.\]
     \item For every $x\in S_X$ the following condition is satisfied
     \[\forall \varepsilon > 0\;\forall \delta > 0\;\exists x_1,x_2\in X:\quad 
     (x_1,x_2)\approximates[\ell_p^2]{1+\varepsilon} \text{ and }\|2^{1/p}\cdot x - x_1 - x_2\| < \delta.\]
\end{enumerate}
\end{theorem}
\begin{proof}
$(i)\implies(ii)$: Pick $f\in S_{L_p[0,1]}$ and $N\in\Nat$. Then, using the continuity of the mapping $[0,1]\ni x\mapsto \int_0^x |f|$, we find $0=x_0<x_1<\ldots<x_N=1$ such that $\int_{x_{i-1}}^{x_i} |f|^p = \frac{1}{N} \int_0^1 |f|^p$ for every $i=1,\ldots,N$. We put $f_i:=N^{1/p}\cdot f\cdot \chi_{[x_{i-1},x_i]}$, $i=1,\ldots,N$. Then, since the supports of $f_i$ are disjoint and since $f_i$ are normalized, we have $(f_i)_{i=1}^N \equiv \ell_p^N$. Further, we obviously have $N^{1/p}\cdot f = \sum_{i=1}^N f_i$.

Obviously, we have $(ii)\implies (iii)$ and $(iii)\implies (iv)$.

$(iii)\implies (i)$: In order to get a contradiction, let us assume that $X$ is not isometric to $L_p[0,1]$ which, by Theorem~\ref{thm:scriptLpCharacterization}, implies that $X$ is isometric to $L_p(\mu)$, where $(\Omega,S,\mu)$ is a measure space for which there is $\omega\in \Omega$ with $\mu(\{\omega\}) = 1$. Fix $\varepsilon > 0$ small enough (to be specified later). Suppose to the contrary that there are $f,g\in L_p(\mu)$ such that $(f,g)\approximates[\ell_p^2]{1+\varepsilon}$ and $2^{\frac 1p}\cdot\delta_{\omega}=f+g$, where $\delta_{\omega}$ is the Dirac function supported by the point $\omega$. For $\mu$-a.e. $x\in\Omega\setminus\{\omega\}$, we have $f(x)+g(x)=0$, so we assume this holds for all $x\in\Omega\setminus\{\omega\}$. We without loss of generality assume that $f(\omega)\geq g(\omega)$.

We claim that both $f(\omega)$ and $g(\omega)$ are positive and $|f(\omega) - g(\omega)|^p<\frac{1}{2}$ if $\varepsilon>0$ is chosen sufficiently small. Indeed, we have
\[(1+\varepsilon)^p - \frac{1}{(1+\varepsilon)^p}\geq \Big|\|f\|_p^p - \|g\|_p^p\Big| = \Big||f(\omega)|^p - |g(\omega)|^p\Big|,\]
which implies $||f(\omega)| - |g(\omega)||<2^{-1/p}$ for sufficiently small $\varepsilon>0$. The claim follows since if both $f(\omega)$ and $g(\omega)$ were not positive we would have $2^{1/p} > 2^{-1/p} > ||f(\omega)| - |g(\omega)|| = |(f+g)(\omega)| = 2^{1/p}$, a contradiction.

First, let us handle the case when $1\leq p<2$. We have
\[\begin{split}
    \|2f\|_p^p & = \int_{\Omega\setminus\{\omega\}} |2f|^p \dmu + (2f(\omega))^p = \int_{\Omega\setminus\{\omega\}} |f-g|^p \dmu + (2f(\omega))^p\\
    & = \|f-g\|_p^p + (2f(\omega))^p - ((f-g)(\omega))^p\\
    & \geq  \|f-g\|_p^p + ((f+g)(\omega))^p = \|f-g\|_p^p + \|f+g\|_p^p,
\end{split}\]
where in the inequality we used superadditivity of the function $[0,\infty)\ni t\mapsto t^p$. Thus, $(f,g)\approximates[\ell_2^p]{1+\varepsilon}$ implies
\[
(1+\varepsilon)^p\geq \|f\|_p^p\geq \frac{\|f-g\|_p^p + \|f+g\|_p^p}{2^p}\geq \frac{4}{2^p(1+\varepsilon)^p};
\]
hence, if $1\leq p<2$ we get a contradiction for sufficiently small $\varepsilon>0$.

Finally, let us handle the case when $p>2$. Note that since $f(\omega)\geq g(\omega)\geq 0$ and $f(\omega) + g(\omega) = 2^{1/p}$, we have $g(\omega)\leq 2^{1/p-1}$. Further, we have
\[\begin{split}
\|2g\|_p^p & = \int_{\Omega\setminus\{\omega\}} |f-g|^p \dmu + (2g(\omega))^p \leq \|f-g\|_p^p + 2.
\end{split}\]
Thus, $(f,g)\approximates[\ell_2^p]{1+\varepsilon}$ implies
\[
\frac{1}{(1+\varepsilon)^p}\leq \|g\|_p^p\leq \frac{\|f-g\|_p^p + 2}{2^p}\leq \frac{2(1+\varepsilon)^p + 2}{2^p};
\]
hence, if $p>2$ we get a contradiction for sufficiently small $\varepsilon>0$.

$(iv)\implies (iii)$: Fix $x\in S_X$ and $\varepsilon > 0$. Pick $\delta>0$ small enough (to be specified later). Applying the condition $(iv)$ we obtain $x_1',x_2'\in X$ such that $(x_1',x_2')\approximates[\ell_p^2]{1+\tfrac{\varepsilon}{2}}$ and $\|2^{1/p}\cdot x - x_1' - x_2'\| < \delta$. Now set $x_1=x'_1+(2^{1/p}\cdot x-(x'_1+x'_2))/2$ and $x_2=x'_2+(2^{1/p}\cdot x-(x'_1+x'_2))/2$. If $\delta$ was chosen sufficiently small, we have $(x_1,x_2)\approximates[\ell_p^2]{1+\varepsilon}$ and clearly $2^{1/p}\cdot x=x_1+x_2$.
\end{proof}

Let us note the following easy observation. The proof is easy and so omitted.

\begin{fact}\label{fact:openCondition}
Let $v,w\in V$, $v\neq 0$ and $a,b\in\Rea$. Then the set  \[\left\{\mu\in\PP\setsep \mu(v)\neq 0\text{ and }\mu(a\cdot\frac{v}{\mu(v)} - w)<b\right\}\]
is open in $\PP$.
\end{fact}

\begin{proof}[Proof of Theorem~\ref{thm:classOfLp}]
Let $\F$ be the set of those $\nu\in\PP_\infty$ for which $X_\nu$ is an $\mathcal{L}_{p,1+}$-space. By Theorem~\ref{thm:scriptLp}, $\F\subseteq \PP_\infty$ is a $G_\delta$-set. By Theorem~\ref{thm:characterizationLp}, using the obvious observation that condition $(iv)$ may be verified on a dense subset, we have
\[\isomtrclass[\PP_\infty]{L_p[0,1]} = \F\cap\bigcap_{v\in V}\bigcap_{n,k\in\Nat} U_{v,n,k},\]
where $U_{v,n,k}$ are open sets (using Fact~\ref{fact:openCondition} and Lemma~\ref{lem:infiniteDimIsGDelta}) defined as
\[U_{v,n,k}:=\Big\{\mu\in\PP_\infty\setsep \exists v_1,v_2\in V\setsep (v_1,v_2)\approximates[\ell_p^2]{1+\tfrac{1}{n}} \text{ and }\mu\big(2^{1/p}\cdot \frac{v}{\mu(v)} - v_1 - v_2\big) < \frac{1}{k}\Big\}.\]
Thus, $\isomtrclass[\PP_\infty]{L_p[0,1]}$ is a $G_\delta$-set.

On the other hand, since any $L_p(\mu)$ is finitely representable in $\ell_p$ and vice versa (see e.g. \cite[Proposition 12.1.8]{albiacKniha}), from Lemma~\ref{lem:uniqueGDelta} and Theorem~\ref{thm:scriptLpCharacterization} we obtain that there is at most one (up to isometry) $\mathcal{L}_{p,1+}$ space $X$ such that $\isomtrclass{X}$ is a $G_\delta$-set in $\B$ and that $\isomtrclass{L_p[0,1]}$ is a $G_\delta$-complete set.
\end{proof}

\section{Spaces with \texorpdfstring{$F_{\sigma \delta}$}{Fsigmadelta} isometry classes}\label{sect:Fsigmadelta}
In this section we focus on another classical Banach spaces, namely $\ell_p$ spaces, for $p\in[1,2)\cup(2,\infty)$, and $c_0$. The main result of this section is the following, which proves the second and the third part of Theorem~\ref{thm:Intro2}.
\begin{theorem}\label{thm:c0AndEllP}
The sets $\isomtrclass{c_0}$ and $\isomtrclass{\ell_p}$ (for $p\in[1,2)\cup(2,\infty)$) are $F_{\sigma \delta}$-complete sets in both $\PP_\infty$ and $\B$.
\end{theorem}
Note that in order to obtain that result we prove Proposition~\ref{prop:characterizationOflp} and Theorem~\ref{thm:characterizationOfc0}, which are of independent interest and where the ``easiest possible'' isometric characterizations of the Banach spaces $\ell_p$, resp. $c_0$, among separable $\mathcal L_{p,1+}$-spaces, resp. separable $\mathcal L_{\infty,1+}$-spaces are given. The proof of Theorem~\ref{thm:c0AndEllP} follows immediately from Proposition~\ref{prop:hardPart}, Proposition~\ref{prop:class_of_l_p} and Proposition~\ref{prop:classOfc0}.

Let us emphasize that in subsection~\ref{subsection:szlenkDerivative} we compute the Borel complexity of the operation assigning to a given Banach space the Szlenk derivative of its dual unit ball, which could be of an independent interest as well. See e.g. subsection~\ref{subsection:szlenkIndices} for some consequences. The reason why we need to do it here is obviously that our isometric characterization of the space $c_0$ involves Szlenk derivatives.

We start with the part which is common for both cases -- that is, for $\isomtrclass{c_0}$ and $\isomtrclass{\ell_p}$.
\begin{lemma}\label{lem:crucialStepHard}
Let $p\in[1,\infty)$ and let $X=(\bigoplus_{n\in\Nat} X_n)_p$ be the $\ell_p$-sum of a family $(X_n)_{n\in\Nat}$ of separable infinite-dimensional Banach spaces. Then $X\equiv \ell_p$ if and only if $X_n\equiv \ell_p$ for every $n\in\Nat$.

Similarly, let $X=(\bigoplus_{n\in\Nat} X_n)_0$ be the $c_0$-sum of a family $(X_n)_{n\in\Nat}$ of separable infinite-dimensional Banach spaces. Then $X\equiv c_0$ if and only if $X_n\equiv c_0$ for every $n\in\Nat$.
\end{lemma}

\begin{proof}
It is easy and well-known that the $\ell_p$-sum of countably many $\ell_p$ spaces is isometric to $\ell_p$, and that the $c_0$-sum of countably many $c_0$ spaces is isometric to $c_0$. The opposite implications follow from the facts that every $1$-complemented infinite-dimensional subspace of $\ell_p$ is isometric to $\ell_p$, and that every $1$-complemented infinite-dimensional subspace of $c_0$ is isometric to $c_0$, see \cite[page 54]{LiTz77}.
\end{proof}

\begin{proposition}\label{prop:hardPart}
Let $X$ be one of the spaces $\ell_p$, $p\in[1,2)\cup(2,\infty)$, or $c_0$.
Then the set $\isomtrclass{X}$ is $F_{\sigma \delta}$-hard in $\B$.
\end{proposition}

\begin{proof}
Our plan is to find a Wadge reduction of a known $F_{\sigma \delta}$-hard set to $\isomtrclass[\B]{X}$.
For this purpose we will use the set
\[
P_3=\{x\in 2^{\Nat\times\Nat}\colon\forall m\text{ there are only finitely many }n\text{'s with } x(m,n)=1 \}
\]
(see e.g. \cite[Section 23.A]{KechrisBook} for the fact that $P_3$ is $F_{\sigma \delta}$-hard in $2^{\Nat\times\Nat}$).
But before we start to construct the reduction of $P_3$ to $\isomtrclass[\B]{X}$ we need to do some preparation.

By Theorem~\ref{thm:classOfLp} (in case $X=\ell_p$) and Theorem~\ref{thm:gurariiTypicalInPGeneralized} (in case $X=c_0$) we know that $\isomtrclass[\B]{X}$ is not a $G_\delta$-set in $\B$. Therefore it is $F_\sigma$-hard in $\B$ by Lemma~\ref{lem:F_sigma-hard}. Now as the set
\[
N_2=\{x\in 2^\Nat\colon\textnormal{there are only finitely many }n\textnormal{'s with }x(n)=1\}
\]
is an $F_\sigma$-set in $2^\Nat$, it is Wadge reducible to $\isomtrclass[\B]{X}$, so there is a continuous function $\varrho\colon 2^\Nat\rightarrow\B$ such that
\[
x\in N_2\Leftrightarrow\varrho(x)\in\isomtrclass[\B]{X}.
\]

We fix a bijection $b\colon\Nat^2\rightarrow\Nat$.
For every $x\in 2^\Nat$ and every $m\in\Nat$ we define $\varrho_m(x)\in\PP_\infty$ as follows.
Suppose that $v=\sum_{n\in\Nat}\alpha_ne_n$ is an element of $V$ (i.e., $\alpha_n$ is a rational number for every $n$, and $\alpha_n\neq 0$ only for finitely many $n$'s), then we put
\[
\varrho_m(x)(v)=\varrho(x)\left(\sum_{n\in\Nat}\alpha_{b(m,n)}e_n\right).
\]
Note that the set $\{e_{b(m,n)}\colon n\in\Nat\}$ is both linearly independent and linearly dense in $X_{\varrho_m(x)}$, and that $\varrho_m(x)(e_k)=0$ if $k\notin\{b(m,n)\colon n\in\Nat\}$. Also, $X_{\varrho_m(x)}$ is isometric to $X_{\varrho(x)}$, where the isometry is induced by the operator $$e_k\mapsto \begin{cases}e_n & k=b(m,n),\\
0 & k\notin\{b(m,n)\colon n\in\Nat\}.
\end{cases}$$

Now we are ready to construct the required reduction $f\colon 2^{\Nat\times\Nat}\rightarrow\B$.
For every $x\in 2^{\Nat\times\Nat}$ and every $m\in\Nat$ we write $x^{(m)}$ for the sequence $(x(m,n))_{n\in\Nat}$.
If $X=\ell_p$, we define
\[
f(x)(v)=\sqrt[p]{\sum_{m\in\Nat}\left(\varrho_m(x^{(m)})(v)\right)^p},\;\;\;\;\;v\in V,
\]
and if $X=c_0$ we put
\[
f(x)(v)=\sup\{(\varrho_m(x^{(m)}))(v)\colon m\in\Nat\},\qquad v\in V.
\]
This formula, together with the preceding considerations, easily imply that $f(x)\in\B$ and that $X_{f(x)}$ is isometric to the $\ell_p$-sum, or to the $c_0$-sum (depending on whether $X=\ell_p$ or $X=c_0$), of the spaces $X_{\varrho(x^{(m)})}$, $m\in\Nat$.
Continuity of the functions $\varrho_m$ and $x\mapsto x^{(m)}$, $m\in\Nat$, immediately implies continuity of $f$. By Lemma~\ref{lem:crucialStepHard}, $f(x)\in \isomtrclass[\B]{X}$ if and only if $\varrho(x^{(m)})\in \isomtrclass[\B]{X}$ for every $m\in\Nat$. Hence, 
\[
x\in P_3\Leftrightarrow \forall m\in\Nat:\; x^{(m)}\in N_2\Leftrightarrow f(x)\in\isomtrclass[\B]{X}.\qedhere
\]
\end{proof}

\subsection{The spaces $\ell_p$}\label{sect:ellP}
The purpose of this subsection is to prove the following result.
\begin{proposition}\label{prop:class_of_l_p}
For every $p\in[1,2)\cup(2,\infty)$ we have that $\isomtrclass{\ell_p}$ is an $F_{\sigma\delta}$-set in $\PP_\infty$.
\end{proposition}

We start with the following classical result, which is sometimes named Clarkson's inequality. The proof may be found on various places, the original one is in the paper by Clarkson, see \cite{C36}. In fact, we use only a very special case of Clarkson's inequality where $z,w$ are required to be elements of the real line instead of an $L_p$ space (and this case is rather straightforward to prove).

\begin{lemma}[Clarkson's inequality]\label{lem:inequalityDisjointSupports}
Let $1\leq p < \infty$, $p\neq 2$. If $p>2$, then for every $z,w\in\Rea$ we have
\[|z+w|^p + |z-w|^p - 2|z|^p - 2|w|^p\geq 0.\]
If $p<2$ then the reverse inequality holds. Moreover, the equality holds if and only if $zw=0$.
\end{lemma}

\begin{proposition}\label{prop:characterizationOflp}
Let $1\leq p <\infty$, $p\neq 2$, and let $X$ be a separable infinite-dimensional $\mathcal{L}_{p,1+}$-space. Let $D$ be a dense subset of $X$. Then the following assertions are equivalent.
\begin{enumerate}[(i)]
    \item $X$ is isometric to $\ell_p$.
    \item For every $x\in S_X$ and every $\delta\in (0,1)$ the following condition is satisfied:
    \[\exists N\in\Nat\; \exists \varepsilon>0\;\forall x_1,\ldots,x_N\in X:\quad (N^{1/p}\cdot x_i)_{i=1}^N\approximates[\ell_p^N]{1+\varepsilon} \Rightarrow \|x-\sum_{i=1}^N x_i\|>\delta.\]
    \item For every $x\in S_X$ the following condition is satisfied:
    \[\exists N\in\Nat\; \forall x_1,\ldots,x_N\in X:\quad (N^{1/p}\cdot x_i)_{i=1}^N\equiv\ell_p^N \Rightarrow x\neq \sum_{i=1}^N x_i.\]
    
    \item For every $x\in D\setminus\{0\}$ and every $\delta\in (0,1)$ the following condition is satisfied:
    \[\exists N\in\Nat\; \exists \varepsilon>0\;\forall x_1,\ldots,x_N\in D:\quad (N^{1/p}\cdot x_i)_{i=1}^N\approximates[\ell_p^N]{1+\varepsilon} \Rightarrow \Big\|\frac x{\|x\|}-\sum_{i=1}^N x_i\Big\|\ge\delta.\]
    
\end{enumerate}
\end{proposition}
\begin{proof}
$(i)\implies(ii)$: Fix $x\in S_{\ell_p}$ and $\delta\in(0,1)$. Pick $l\in\Nat$ with $\sum_{k=1}^l |x(k)|^p > \delta^p$ and $N\in\Nat$ such that $\sum_{k=1}^l (|x(k)|  - \frac{3}{\sqrt[p]{N}})^p > \delta^p$. Fix a sequence $(\varepsilon_m)_{m\in\Nat}\in(0,1)^\Nat$ with $\varepsilon_m\to 0$. In order to get a contradiction, for every $m\in\Nat$, pick $x_1^{\varepsilon_m},\ldots,x_N^{\varepsilon_m}\in \ell_p$ such that $(N^{1/p}\cdot x_i^{\varepsilon_m})_{i=1}^N\approximates[\ell_p^N]{1+\varepsilon_m}$ and $\|x-\sum_{i=1}^N x_i^{\varepsilon_m}\|\leq \delta$ for every $m\in\Nat$.

We \emph{claim} that there is $m\in\Nat$ such that $|x_i^{\varepsilon_m}(k)x_j^{\varepsilon_m}(k)|<\eta:=N^{-(2+2/p)}$ for every $i,j\in\{1,\ldots,N\}$, $i\neq j$, and $k\in\{1,\ldots,l\}$. Indeed, otherwise there are $i,j,k$ such that $|x_i^{\varepsilon_m}(k)x_j^{\varepsilon_m}(k)|\geq \eta$ for infinitely many $m$'s. By passing to a subsequence, we may assume that this holds for every $m\in\Nat$. Since the sequences $(|x_i^{\varepsilon_m}(k)|)_m$  and $(|x_j^{\varepsilon_m}(k)|)_m$ are bounded, by passing to a subsequence we may assume there are numbers $a,b\in\Rea$ with $x_i^{\varepsilon_m}(k)\to a$, $x_j^{\varepsilon_m}(k)\to b$ and $|ab|\geq \eta > 0$. Since  $(N^{1/p}\cdot x_i^{\varepsilon_m},N^{1/p}\cdot x_j^{\varepsilon_m})\approximates[\ell_p^2]{1+\varepsilon_m}$, using Lemma~\ref{lem:inequalityDisjointSupports}, for $p>2$ we obtain
\[\begin{split}
    0 & \leq |a+b|^p + |a-b|^p - 2|a|^p - 2|b|^p\\
    & = \lim_m
    \big(|x_i^{\varepsilon_m}(k) + x_j^{\varepsilon_m}(k)|^p + |x_i^{\varepsilon_m}(k) - x_j^{\varepsilon_m}(k)|^p - 2|x_i^{\varepsilon_m}(k)|^p - 2|x_j^{\varepsilon_m}(k)|^p\big)\\
    & \leq \lim_m \big(\|x_i^{\varepsilon_m} + x_j^{\varepsilon_m}\|^p + \|x_i^{\varepsilon_m} - x_j^{\varepsilon_m}\|^p - 2\|x_i^{\varepsilon_m}\|^p - 2\|x_j^{\varepsilon_m}\|^p\big) = 0;
\end{split}\]
hence, $|a+b|^p + |a-b|^p=2|a|^p+2|b|^p=0$ which, by Lemma~\ref{lem:inequalityDisjointSupports}, is in contradiction with $|ab|>0$. The case when $p<2$ is similar.

From now on, we write $x_i$ instead of $x_i^{\varepsilon_m}$, where $m\in\Nat$ is chosen to satisfy the claim above. Fix $k\leq l$. By the claim above, there is at most one $i_0\in\{1,\ldots,N\}$ with $|x_{i_0}(k)| \geq \sqrt{\eta}$ and for this $i_0$ we have $|x_{i_0}(k)|\leq \|x_{i_0}\|\leq 2N^{-1/p}$. Consequently, we have
\[
\sum_{i=1}^N |x_i(k)|\leq \frac{2}{N^{1/p}} + \sum_{i\in \{1,\ldots,N\}\& |x_i(k)|<\sqrt{\eta}} |x_i(k)|\leq \frac{2}{N^{1/p}} + N\cdot\sqrt{\eta} = \frac{3}{N^{1/p}}.
\]

Thus, we have
\[\|x-\sum_{i=1}^N x_i\|^p\geq \sum_{k=1}^l \Big(|x(k)| - \sum_{i=1}^N |x_i(k)|\Big)^p\geq \sum_{k=1}^l \Big(|x(k)| - \frac{3}{N^{1/p}}\Big)^p > \delta^p,\]
which is in contradiction with $\|x-\sum_{i=1}^N x_i\| = \|x-\sum_{i=1}^N x_i^{\varepsilon_m}\|\leq \delta$.

$(ii)\implies (iii)$ is obvious.

$(iii)\implies (i)$: Suppose that $X$ is not isometric to $\ell_p$.
By Theorem~\ref{thm:scriptLpCharacterization}, $X$ is isometric to $L_p[0,1]\oplus_p Y$ for some (possibly trivial) Banach space $Y$.
By abusing the notation, we may assume that $X=L_p[0,1]\oplus_p Y$.
Let ${\bf 1}\in L_p[0,1]$ be the constant 1 function, and define $x\in X=L_p[0,1]\oplus_p Y$ by $x=({\bf 1},0)$.
Now fix $N\in\Nat$ arbitrarily.
Define $x_1,\ldots,x_N\in X$ by $x_i=(\chi_{[\frac{i-1}n,\frac in]},0)$. Clearly $(N^{1/p}\cdot x_i)_{i=1}^N\equiv\ell^n_p$ and we have $x = \sum_{i=1}^N x_i$.

$(ii)\implies(iv)$ is obvious, so it only remains to show that $(iv)\implies(ii)$. For every $x\in X\setminus\{0\}$, $\delta\in(0,1)$, $N\in\Nat$, $\varepsilon>0$ and $x_1,\ldots,x_N\in X$ we denote by $V(x,\delta,N,\varepsilon,(x_i)_{i=1}^N)$ the assertion that if $(N^{1/p}\cdot x_i)_{i=1}^N\approximates[\ell_p^N]{1+\varepsilon}$ then $\|\frac x{\|x\|}-\sum_{i=1}^N x_i\|\ge\delta$. The desired implication straightforwardly follows by the following two easy observations. First, if $x\in D\setminus\{0\}$, $\delta$, $N$ and $\varepsilon$ are given such that $V(x,\delta,N,\varepsilon,(x_i)_{i=1}^N)$ holds for every $x_1,\ldots,x_N\in D$ then $V(x,\delta,N,\varepsilon,(x_i)_{i=1}^N)$ holds for every $x_1,\ldots,x_N\in X$. Second, if for every $x\in D\setminus\{0\}$ and $\delta$ there are $N$ and $\varepsilon$ such that $V(x,\frac{1+\delta}2,N,\varepsilon,(x_i)_{i=1}^N)$ holds for every $x_1,\ldots,x_N\in X$, then for every $x\in X\setminus\{0\}$ and $\delta$ there are $N$ and $\varepsilon$ such that $V(x,\delta,N,\varepsilon,(x_i)_{i=1}^N)$ holds for every $x_1,\ldots,x_N\in X$.
\end{proof}

\begin{proof}[Proof of Proposition~\ref{prop:class_of_l_p}]
Let $\F$ be the set of those $\nu\in\PP_\infty$ for which $X_\nu$ is an $\mathcal{L}_{p,1+}$-space. By Theorem~\ref{thm:scriptLp}, $\F\subseteq \PP_\infty$ is a $G_\delta$-set. By Proposition~\ref{prop:characterizationOflp} $(i)\Leftrightarrow(iv)$, we have
\[\isomtrclass[\PP_\infty]{\ell_p} = \F\cap\bigcap_{v\in V\setminus\{0\}}\bigcap_{m\in\Nat}\bigcup_{n,k\in\Nat}V_{v,m,n,k},\]
where the closed (see Fact~\ref{fact:openCondition} and Lemma~\ref{lem:infiniteDimIsGDelta}) sets $V_{v,m,n,k}$ are given by
\[\begin{split}
V_{v,m,n,k}:=\Big\{\mu\in\PP_\infty\setsep & \mu(v)=0\text{ or for every $(v_i)_{i=1}^n\in V^n$ we have}\\
& \neg\left((\sqrt[p]nv_i)_{i=1}^n\approximates[\ell_p^n]{1+\frac 1k}\right) \text{ or }\mu\big(\frac{v}{\mu(v)} - \sum_{i=1}^nv_i\big)\ge\frac 1m\Big\}.
\end{split}\]
Thus, $\isomtrclass[\PP_\infty]{\ell_p}$ is an $F_{\sigma\delta}$-set.
\end{proof}

\subsection{Dual unit balls and the Szlenk derivative}\label{subsection:szlenkDerivative} The purpose here is to show that mappings which assign the dual unit ball and its Szlenk derivative to a separable Banach space may be realized as Borel maps, see Lemma~\ref{lem:slozitost_koule} and Lemma~\ref{lem:slozitost_derivace}. This will be later used in order to estimate the Borel complexity of the isometry class of the space $c_0$ because the isometric characterization of the space $c_0$ we use involves Szlenk derivatives, see Theorem~\ref{thm:characterizationOfc0}. Note that the issue of handling Szlenk derivations as Borel maps was previously considered also by Bossard in \cite[page 141]{Bo02}, but our approach is slightly different as we prefer to work with the coding $\PP$ and we also need to obtain an estimate on the Borel class of the mapping.

Let us recall that given a real Banach space $X$, a $w^*$-compact set $F\subseteq X^*$ and $\varepsilon>0$, the Szlenk derivative is given as
\[
F'_{\varepsilon} = \big\{ x^{*} \in F : U \ni x^{*} \textnormal{ is $ w^{*} $-open} \Rightarrow \mathrm{diam}(U \cap F) \geq \varepsilon \big\}.
\]

We start by coding dual unit balls as closed subsets of $ B_{\ell_{\infty}} $ equipped with the weak* topology, i.e., the topology generated by elements of the unique predual $\ell_1$.

\begin{lemma}\label{lem:ballIsOmega}
Let $X$ be a separable Banach space and let $\{x_n\colon n\in\Nat\}$ be a dense set in $B_X$. Then the mapping $B_{X^{*}}\ni x^*\mapsto (x^*(x_n))_{n=1}^\infty\in B_{\ell_\infty}$ is $\|\cdot\|$-$\|\cdot\|$ isometry and $w^*$-$w^*$ homeomorphism onto the set
\[
\Omega(X):=\Big\{ (a_{n})_{n=1}^{\infty} \in B_{\ell_{\infty}} : M \subseteq \mathbb{N} \textnormal{ finite} \Rightarrow  \Big| \sum_{n \in M} a_{n} \Big| \leq \Big\Vert \sum_{n \in M} x_{n} \Big\Vert \Big\}.
\]
\end{lemma}
\begin{proof}
That the mapping is $w^*$-$w^*$ homeomorphism onto its image follows from the fact that $B_{X^*}$ is $w^*$-compact and the mapping is one-to-one (because $(x_n)$ separate the points of $B_{X^*}$) and $w^*$-$w^*$ continuous (because on $B_{\ell_\infty}$ the $w^*$-topology coincides with the topology of pointwise convergence). It is also straightforward to see that the mapping is isometry. Thus, it suffices to prove that
\[
\{(x^*(x_n))_{n=1}^\infty\colon x^*\in B_{X^{*}}\} = \Omega(X).
\]
The inclusion $ \subseteq $ is easy, let us prove $ \supseteq $. Given numbers $ a_{1}, a_{2}, \dots $ satisfying $ | \sum_{n \in M} a_{n} | \leq \Vert \sum_{n \in M} x_{n} \Vert $ for any finite $ M \subseteq \mathbb{N} $, we need to find $ x^{*} \in B_{X^{*}} $ such that $ x^{*}(x_{n}) = a_{n} $ for each $ n $.

Let us realize first that
\begin{itemize}
\item $ |a_{n} - a_{m}| \leq \Vert x_{n} - x_{m} \Vert $ for every $ n, m $,
\item $ |a_{n} + a_{m} - a_{l}| \leq \Vert x_{n} + x_{m} - x_{l} \Vert $ for every $ n, m, l $.
\end{itemize}
We check the first inequality only, the second inequality can be checked in the same way. Given $ \varepsilon > 0 $, let $ n' $ different from $ n $ and $ m $ be such that $ \Vert x_{n} + x_{n'} \Vert < \varepsilon $. We obtain $ |a_{n} - a_{m}| = |(a_{n} + a_{n'}) - (a_{m} + a_{n'})| \leq |a_{n} + a_{n'}| + |a_{m} + a_{n'}| \leq \Vert x_{n} + x_{n'} \Vert + \Vert x_{m} + x_{n'} \Vert \leq 2 \Vert x_{n} + x_{n'} \Vert + \Vert x_{m} - x_{n} \Vert < 2 \varepsilon + \Vert x_{m} - x_{n} \Vert $. Since $ \varepsilon > 0 $ was chosen arbitrarily, we arrive at $ |a_{n} - a_{m}| \leq \Vert x_{m} - x_{n} \Vert $.

It follows that there is a function $ f : B_{X} \to \mathbb{R} $ with the Lipschitz constant $ 1 $ such that $ f(x_{n}) = a_{n} $ for each $ n $. We claim that $ f(u + v) = f(u) + f(v) $ and $f(\alpha u)=\alpha f(u)$, whenever $ u, \alpha u, v, u + v \in B_{X} $. Given $ \varepsilon > 0 $, we pick $ n, m, l $ such that $ \Vert x_{n} - u \Vert < \varepsilon, \Vert x_{m} - v \Vert < \varepsilon $ and $ \Vert x_{l} - (u + v) \Vert < \varepsilon $. Then $ |f(u) + f(v) - f(u + v)| < |a_{n} + a_{m} - a_{l}| + 3\varepsilon \leq \Vert x_{n} + x_{m} - x_{l} \Vert + 3\varepsilon \leq \Vert u + v - (u + v) \Vert + 3\varepsilon + 3\varepsilon = 6\varepsilon $. Since $ \varepsilon > 0 $ was chosen arbitrarily, we arrive at $ |f(u) + f(v) - f(u + v)| = 0 $. This also shows that $f(u/2)=f(u)/2$, therefore $f(\alpha u)=\alpha f(u)$, provided that $\alpha$ is a dyadic rational number. For a general $\alpha$, we use density of dyadic rationals and continuity of $f$.

Now, it is easy to see that $f$ uniquely extends to a linear functional on $X$.
\end{proof}

By the above, every dual unit ball of a separable Banach space may be realized as a subset of the unit ball of $\ell_\infty$. Thus, in what follows we use the following convention.

\begin{convention}
Whenever we talk about open (closed, $F_\sigma$, etc.) subsets of $B_{l_\infty}$ we always mean open (closed, $F_\sigma$, etc.) subsets in the weak* topology. On the other hand, whenever we talk about the diameter of a subset of $B_{l_\infty}$, or about the distance of two subsets of $B_{l_\infty}$, we always mean the diameter, or the distance, with respect to the metric given by the norm of $\ell_\infty$. Also, we write only $\mathcal K(B_{l_\infty})$ instead of $\mathcal K(B_{l_\infty},w^*)$.
\end{convention}

Let us note the following easy observation for further references.
\begin{lemma}\label{lem:meritelnost}
Let $P$ be a Polish space, $X$ a metrizable compact, $\alpha\in[1,\omega_1)$ and $f:P\to\mathcal K(X)$ a mapping such that $\{p\in P\colon f(p)\subseteq W\}\in \boldsymbol{\Sigma}_\alpha^0(P)\cup\boldsymbol{\Pi}_\alpha^0(P)$ for every open $W\subseteq X$. Then $f$ is $\boldsymbol{\Sigma}_{\alpha+1}^0$-measurable.
\end{lemma}
\begin{proof}
The sets of the form
\[
\{F\in\mathcal K(X)\colon F\subseteq W\}\;\;\;\text{ and }\;\;\;\{F\in\mathcal K(X)\colon F\cap W\neq\emptyset\},
\]
where $W$ ranges over all open subsets of $X$, form a subbasis of the topology of $\mathcal K(X)$. So we only need to check that $f^{-1}(U)$ is a $\boldsymbol{\Sigma}_{\alpha+1}^0$-set for every open set $U$ of one of these forms. For the first case this follows immediately from the assumptions and for the second case, if $\{W_n\colon n\in\Nat\}$ is an open basis for the topology of $X$, we have
\[
f^{-1}(\{F\in\mathcal K(X)\colon F\cap W\neq\emptyset\}) = \bigcup _{n\in\Nat \text{ such that $\overline{W_n}\subseteq W$}} P\setminus \{p\in P\colon f(p)\subseteq X\setminus \overline{W_n}\},
\]
which, by the assumptions, is the countable union of sets from  $\boldsymbol{\Sigma}_\alpha^0(P)\cup\boldsymbol{\Pi}_\alpha^0(P)$.
\end{proof}

\begin{lemma}\label{lem:slozitost_koule}
For every $\nu\in\mathcal P$ we can choose a countable dense subset $\{x_n^\nu\colon n\in\mathbb N\}$ of $B_{X_\nu}$ in such a way that the mapping $\Omega\colon\mathcal P\rightarrow\mathcal K(B_{l_\infty},w^*)$ given by
\[
\Omega(\nu)=\left\{(a_n)_{n=1}^\infty\in B_{l_\infty}\colon M\subseteq\mathbb N\textnormal{ finite}  \Rightarrow\left|\sum\limits_{n\in M}a_n\right|\le \nu\left(\sum\limits_{n\in M}x_n^\nu\right)\right\}
\]
is continuous.
\end{lemma}

\begin{proof}
First of all, we describe the choice of the sets $\{x_n^\nu\colon n\in\mathbb N\}$, $\nu\in\mathcal P$. Let $g\colon[0,\infty)\rightarrow[1,\infty)$ be given by $g(t)=1$ for $t\le 1$ and $g(t)=t$ for $t>1$. Let $\{v_n\colon n\in\mathbb N\}$ be an enumeration of all elements of the vector space $V$ (which is naturally embedded into all Banach spaces $X_\nu$, $\nu\in\mathcal P$). Now for every $\nu\in\mathcal P$ and every $n\in\mathbb N$ we define $x_n^\nu\in B_{X_\nu}$ by $x_n^\nu=\frac {v_n}{g(\nu(v_n))}$. Then for every $\nu\in\mathcal P$ we have that $\{x_n^\nu\colon n\in\mathbb N\}$ is a dense subset of $B_{X_\nu}$. Note also that the set
\[
\left\{\left(\nu,(a_n)_{n=1}^\infty\right)\in\mathcal P\times B_{l_\infty}\colon\left|\sum\limits_{n\in M}a_n\right|>\nu\left(\sum\limits_{n\in M}x_n^\nu\right)\right\}
\]
is open in $\mathcal P\times(B_{l_\infty},w^*)$ for every $M\subseteq \Nat$ finite (the proof is easy and is omitted).

Pick an open subset $U$ of $B_{\ell_{\infty}}$. We have
\begin{align*}
&\Omega^{-1}\left(\{F\in\mathcal K(B_{l_\infty})\colon F\subseteq U\}\right)\\
=&\left\{\nu\in\mathcal P\colon\mathop{\forall}\limits_{(a_n)_{n=1}^\infty\in B_{l_\infty}}\left(\left(\mathop{\exists}\limits_{\substack{M\subseteq\mathbb N\\ \text{finite}}}\left|\sum\limits_{n\in M}a_n\right|>\nu\left(\sum\limits_{n\in M}x_n^\nu\right)\right)\text{ or }\left((a_n)_{n=1}^\infty\in U\right) \right)\right\}.
\end{align*}
The complement of the last set is the projection of a closed subset of $\mathcal P\times(B_{l_\infty},w^*)$ onto the first coordinate. As the space $(B_{l_\infty},w^*)$ is compact, the complement is a closed subset of $\mathcal P$.

It remains to show that the set $\{\nu\in\PP\colon \Omega(\nu)\cap U\neq\emptyset\}$ is open. Pick $\nu\in \PP$ with $\Omega(\nu)\cap U\neq\emptyset$. By Lemma~\ref{lem:ballIsOmega}, there exists $x^*\in B_{X_\nu^*}$ such that the sequence $(a_n)_{n=1}^\infty$ given by $a_n=x^*(x_n^\nu)$, $n\in\Nat$, satisfies $(a_n)_{n=1}^\infty\in\Omega(\nu)\cap U$. Let $\varepsilon>0$ and $N\in\Nat$ be such that $(b_n)_{n=1}^\infty\in \ell_\infty$ is an element of $U$ whenever $|b_n-a_n|<\varepsilon$ for every $1\leq n\leq N$. Let us consider subspaces of $c_{00}$ given as $E=\Span\{v_1,\ldots,v_N\}$ and $F=\{x\in E\colon \nu(x)=0\}$. Let $G$ be such that $F\oplus G = E$ and $\overline{G\cap V}=G$ (it is enough to pick a basis of $E$ consisting of vectors from $V$, and using the Gauss elimination to determine which vectors from the basis generate the algebraic complement to $F$). Let $P_F:E\to F$ and $P_G:E\to G$ be linear projections onto $F$ and $G$, respectively.  Pick $\delta < \min\{1,\tfrac{\varepsilon}{3}\}$ such that $ \delta \cdot | x^{*}(P_{G}v_{n}) | < \varepsilon / 3 $ for every $1\leq n\leq N$. Finally, put
\[\begin{split}
\O:=& \{\nu'\in\PP\colon \frac{1}{1-\delta}\nu(x)>\nu'(x)>(1-\delta)\nu(x)\text{ for every $x\in G\setminus \{ 0 \}$}\}\cap\\
& \qquad \bigcap_{n=1}^N\{\nu'\in\PP\colon \nu'(P_{F}v_{n}) < \delta, | \nu'(v_{n}) - \nu(v_{n}) | < \delta\}.
\end{split}\]
Then $\O$ is an open neighborhood of $\nu$, which easily follows from Lemma~\ref{lem:infiniteDimIsGDelta} and the fact that $G\cap V$ is dense in $G$.

We will show that $\O\subseteq \{\nu'\in\PP\colon \Omega(\nu')\cap U\neq\emptyset\}$. Pick $\nu'\in\O$. If we put $y^*(x):=(1-\delta)x^*(x)$ for $x\in G$, then $|y^*(x)|=(1-\delta)|x^*(x)| \leq (1 - \delta) \nu(x) \leq \nu'(x) $ for every $ x \in G $ and so by the Hahn-Banach theorem we may extend $y^*$ to a functional (denoted again by $y^*$) from the dual unit ball of $X_{\nu'}$. By Lemma~\ref{lem:ballIsOmega}, the sequence $(b_n)_{n=1}^\infty$ given by $b_n:=y^*(x_n^{\nu'})$, $n\in\Nat$, is in $\Omega(\nu')$. Moreover, for every $1\leq n\leq N$ we have
\[\begin{split}
|b_n-a_n| & = \big| \tfrac{1}{g(\nu'(v_{n}))} y^{*}(v_{n}) - \tfrac{1}{g(\nu(v_{n}))} x^{*}(v_{n}) \big|\\
& = \big| \tfrac{1}{g(\nu'(v_{n}))} y^{*}(P_{F}v_{n}) + \tfrac{1}{g(\nu'(v_{n}))} y^{*}(P_{G}v_{n}) - \tfrac{1}{g(\nu(v_{n}))} x^{*}(P_{F}v_{n} + P_{G}v_{n}) \big|\\
& = \big| \tfrac{1}{g(\nu'(v_{n}))} y^{*}(P_{F}v_{n}) + \tfrac{1}{g(\nu'(v_{n}))} (1-\delta)x^*(P_{G}v_{n}) - \tfrac{1}{g(\nu(v_{n}))} x^{*}(P_{G}v_{n}) \big|\\
& \leq \tfrac{1}{g(\nu'(v_{n}))} |y^{*}(P_{F}v_{n})| + \big| \tfrac{1}{g(\nu'(v_{n}))} (1 - \delta) - \tfrac{1}{g(\nu(v_{n}))} \big| | x^{*}(P_{G}v_{n}) |\\
& \leq \delta + \big(|g(\nu(v_{n})) - g(\nu'(v_{n}))| + \delta\big)|x^{*}(P_{G}v_{n})|\\
& \leq \delta + 2\delta |x^*(P_{G}v_{n})| < \varepsilon,
\end{split}\]
and so $(b_n)_{n=1}^\infty\in\Omega(\nu')\cap U$. Hence, $\O\subseteq \{\nu'\in\PP\colon \Omega(\nu')\cap U\neq\emptyset\}$, so $\{\nu\in\PP\colon \Omega(\nu)\cap U\neq\emptyset\}$ is an open set and $\Omega$ is a continuous mapping.
\end{proof}

We close this subsection by realizing that the mapping which assigns to every compact subset of $B_{\ell_\infty}$ its Szlenk derivative is Borel. Let us note that the result is almost optimal as the mapping from Lemma~\ref{lem:slozitost_derivace} is not $F_\sigma$-measurable, see Corollary~\ref{cor:lowerBoundDerivative}.

\begin{lemma}
\label{lem:slozitost_derivace}
For every $\varepsilon>0$, the function $s_\varepsilon\colon\mathcal K(B_{l_\infty},w^*)\rightarrow\mathcal K(B_{l_\infty},w^*)$ given by $s_\varepsilon(F)=F'_\varepsilon$ is $\boldsymbol{\Sigma}_3^0$-measurable.
\end{lemma}
\begin{proof}
First, we claim that the set
\[
\{F\in\mathcal K(B_{l_\infty})\colon\text{diam}(U\cap F)<\varepsilon\}
\]
is an $F_\sigma$-set for every open subset $U$ of $B_{l_\infty}$.
Indeed, the set above equals
\begin{align*}
&\bigcup_{k=1}^\infty\{F\in\mathcal K(B_{l_\infty})\colon\text{diam}(U\cap F)\le\varepsilon-\tfrac 1k\}\\
=&\bigcup_{k=1}^\infty\;\bigcap_{\substack{O_1,O_2\text{ open subsets of }U\\ \text{dist}(O_1,O_2)\ge\varepsilon-\tfrac 1k}}\;\{F\in\mathcal K(B_{l_\infty})\colon F\cap O_1=\emptyset\text{ or }F\cap O_2=\emptyset\},
\end{align*}
and our claim immediately follows.

Now let $W$ be an open subset of $B_{l_\infty}$. Let $\{U_n\colon n\in\mathbb N\}$ be an open basis for the weak* topology of $B_{l_\infty}$. Then we have (using a compactness argument in the last equality) that 
\begin{align*}
&s_\varepsilon^{-1}(\{F\in\mathcal K(B_{l_\infty})\colon F\subseteq W\})\\
=&\{F\in\mathcal K(B_{l_\infty})\colon\mathop{\exists}\limits_{M\subseteq\mathbb N}\big(\big(\mathop{\forall}\limits_{n\in M}\text{diam}(U_n\cap F)<\varepsilon\big)\text{ and }\big(F\subseteq W\cup\bigcup_{n\in M}U_n\big)\big)\}\\
=&\{F\in\mathcal K(B_{l_\infty})\colon\mathop{\exists}\limits_{\substack{M\subseteq\mathbb N\\ \text{finite}}}\big(\big(\mathop{\forall}\limits_{n\in M}\text{diam}(U_n\cap F)<\varepsilon\big)\text{ and }\big(F\subseteq W\cup\bigcup_{n\in M}U_n\big)\big)\},
\end{align*}
and our previous claim implies that the last set is an $F_\sigma$-set.

Thus, by Lemma~\ref{lem:meritelnost}, the mapping $s_\varepsilon$ is $\boldsymbol{\Sigma}_3^0$-measurable.
\end{proof}

\subsection{The space $c_0$}\label{sect:c0}
The main goal of this subsection is to prove the following.
\begin{proposition}\label{prop:classOfc0}
$\isomtrclass{c_0}$ is an $F_{\sigma \delta}$-set in $\PP_\infty$.
\end{proposition}

Our estimate on the Borel complexity of the isometry class of $c_0$ is based on an isometric characterization of $c_0$ among $\mathcal{L}_{\infty,1+}$-spaces. Let us recall that $\mathcal{L}_{\infty,1+}$-spaces are often called the \emph{Lindenstrauss spaces} or \emph{$L_1$ predual spaces}. There are many different characterizations of this class of spaces. Let us recall one which we will use further, see e.g. \cite[p. 232]{LaceyBook} (the ``in particular'' part follows from the easy part of Theorem~\ref{thm:scriptLpCharacterization} applied to $X^*$ and the fact that $L_1[0,1]$ is not isomorphic to a subspace of a separable dual Banach space, see e.g. \cite[Theorem 6.3.7]{albiacKniha}).

\begin{theorem}\label{thm:scriptLInfty}
Let $X$ be a Banach space. Then the following conditions are equivalent.
\begin{enumerate}[(i)]
    \item $X$ is an $\mathcal{L}_{\infty,1+}$-space.
    \item $X^*$ is isometric to $L_1(\mu)$ for some measure $\mu$.
\end{enumerate}
In particular, if $X$ is an $\mathcal{L}_{\infty,1+}$-space with $X^*$ separable then $X^*$ is isometric to $\ell_1$.
\end{theorem}

The isometric characterization of $c_0$ which we use for our upper estimate follows.

\begin{theorem}\label{thm:characterizationOfc0}
Let $ X $ be a separable $ \mathcal{L}_{\infty, 1+} $-space and let $ 0 < \varepsilon < 1 $. Then $ X $ is isometric to $ c_{0} $ if and only if
$$ (B_{X^{*}})'_{2\varepsilon} = (1 - \varepsilon) B_{X^{*}}. $$
\end{theorem}

\begin{proof}
First, we show that $ (B_{c_{0}^{*}})'_{2\varepsilon} = (1 - \varepsilon) B_{c_{0}^{*}} $ (this must be known but we were unable to find any reference). By a standard argument, $ (1 - \varepsilon) B_{X^{*}} \subseteq (B_{X^{*}})'_{2\varepsilon} $ for any infinite-dimensional $ X $. (Let $ x^{*} \in (1 - \varepsilon) B_{X^{*}} $. Any $ w^{*} $-open set $ U $ containing $ 0 $ contains also both $ y^{*} $ and $ -y^{*} $ for some $ y^{*} \in S_{X^{*}} $, and so $ \mathrm{diam}(U \cap B_{X^{*}}) = 2 $. For this reason, any $ w^{*} $-open set $ V $ containing $ x^{*} $ fulfills $ \mathrm{diam}(V \cap (x^{*} + \varepsilon B_{X^{*}})) = 2\varepsilon $, in particular, $ \mathrm{diam}(V \cap B_{X^{*}}) \geq 2\varepsilon $. This proves that $ x^{*} \in (B_{X^{*}})'_{2\varepsilon} $.)

Let us show that the opposite inclusion takes place for $ X = c_{0} $. Assuming $ 1 - \varepsilon < \Vert x^{*} \Vert \leq 1 $, we need to check that $ x^{*} \notin (B_{c_{0}^{*}})'_{2\varepsilon} $. Let $ e_{1}, e_{2}, \dots $ be the canonical basis of $ c_{0} $. Let $ n $ be large enough that $ \sum_{i=1}^{n} |x^{*}(e_{i})| > 1 - \varepsilon $ and let $ \delta > 0 $ satisfy $ 2 \delta n < [\sum_{i=1}^{n} |x^{*}(e_{i})|] - (1 - \varepsilon) $. Let
$$ U = \{ y^{*} \in c_{0}^{*} : 1 \leq i \leq n \Rightarrow |y^{*}(e_{i}) - x^{*}(e_{i})| < \delta \}. $$
For $ y^{*}, z^{*} \in U \cap B_{c_{0}^{*}} $, we have
$$ \sum_{i=n+1}^{\infty} |y^{*}(e_{i})| = \Vert y^{*} \Vert - \sum_{i=1}^{n} |y^{*}(e_{i})| \leq 1 - \sum_{i=1}^{n} |x^{*}(e_{i})| + \delta n, $$
and the same for $ z^{*} $, thus
$$ \Vert y^{*} - z^{*} \Vert \leq \sum_{i=1}^{n} |y^{*}(e_{i}) - z^{*}(e_{i})| + \sum_{i=n+1}^{\infty} |y^{*}(e_{i})| + \sum_{i=n+1}^{\infty} |z^{*}(e_{i})| $$
$$ \leq 2 \delta n + 2 \Big[ 1 - \sum_{i=1}^{n} |x^{*}(e_{i})| \Big] + 2 \delta n. $$
We get $ \mathrm{diam}(U \cap B_{c_{0}^{*}}) \leq 4 \delta n + 2 [ 1 - \sum_{i=1}^{n} |x^{*}(e_{i})| ] < 2 \varepsilon $.

Now, let us assume that $ X $ satisfies $ (B_{X^{*}})'_{2\varepsilon} = (1 - \varepsilon) B_{X^{*}} $. Clearly, $ X $ is infinite-dimensional, as $ (B_{X^{*}})'_{2\varepsilon} $ is non-empty. Moreover, $ X^{*} $ is separable because the Szlenk index of $X$ is $ \omega $, see e.g. \cite[Proposition 3 and Theorem 1]{L06}. Thus, by Theorem~\ref{thm:scriptLInfty}, the dual $ X^{*} $ is isometric to $ \ell_{1} $. Let $ e^{*}_{1}, e^{*}_{2}, \dots $ be a basis of $ X^{*} $ that is $ 1 $-equivalent to the canonical basis of $ \ell_{1} $, and let $ e^{**}_{1}, e^{**}_{2}, \dots $ be the dual basic sequence in $ X^{**} $. We claim that the functionals $ e^{**}_{n} $ are $ w^{*} $-continuous.

Suppose that $ e^{**}_{n} $ is not $ w^{*} $-continuous for some $ n $. It means that $ \{ x^{*} \in X^{*} : e^{**}_{n}(x^{*}) = 0 \} $ is not $ w^{*} $-closed. By the Banach-Dieudonn\'e theorem, the set $ \{ x^{*} \in B_{X^{*}} : e^{**}_{n}(x^{*}) = 0 \} $ is not $ w^{*} $-closed, too. The space $ (B_{X^{*}}, w^{*}) $ is metrizable, so there is a sequence $ x^{*}_{k} $ in $ B_{X^{*}} $ with $ e^{**}_{n}(x^{*}_{k}) = 0 $ which $ w^{*} $-converges to some $ x^{*} $ with $ e^{**}_{n}(x^{*}) \neq 0 $. Without loss of generality, let us assume that $ e^{**}_{n}(x^{*}) > 0 $ and that $ e^{**}_{i}(x^{*}_{k}) $ converges to some $ a_{i} $ for every $ i $. Then clearly $ a_{n} = 0 $. Note that $ \sum_{i=1}^{\infty} |a_{i}| \leq 1 $, which follows from the fact that $ \sum_{i=1}^{\infty} |e^{**}_{i}(x^{*}_{k})| = \Vert x^{*}_{k} \Vert \leq 1 $ for every $ k $. Let us put $ a^{*} = \sum_{i=1}^{\infty} a_{i} e^{*}_{i} $, $ y^{*}_{k} = x^{*}_{k} - a^{*} $ and $ y^{*} = x^{*} - a^{*} $. Then $ e^{**}_{n}(y^{*}_{k}) = 0, e^{**}_{n}(y^{*}) > 0 $, the sequence $ y^{*}_{k} $ is $ w^{*} $-convergent to $ y^{*} $ and, moreover, $ e^{**}_{i}(y^{*}_{k}) $ converges to $ 0 $ for every $ i $. Choosing a subsequence and making a small perturbation, we can find a sequence $ z^{*}_{l} $ which is a block sequence with respect to the basis $ e^{*}_{i} $ and which still $ w^{*} $-converges to $ y^{*} $. Without loss of generality, let us assume that $ \Vert z^{*}_{l} \Vert $ converges to some $ \lambda $, clearly with $ \lambda \geq \Vert y^{*} \Vert > 0 $, and let us consider $ u^{*}_{l} = \frac{1}{\Vert z^{*}_{l} \Vert} z^{*}_{l} $ and $ u^{*} = \frac{1}{\lambda} y^{*} $.

So, we have seen that there is a normalized block sequence $ u^{*}_{l} $ in $ X^{*} $ which $ w^{*} $-converges to some $ u^{*} $ with $ e^{**}_{n}(u^{*}) > 0 $. We put
$$ v^{*}_{l} = (1 - \varepsilon) e^{*}_{n} + \varepsilon u^{*}_{l}, \quad v^{*} = (1 - \varepsilon) e^{*}_{n} + \varepsilon u^{*}. $$
Then $ v^{*}_{l} $ is a sequence in $ B_{X^{*}} $ that $ w^{*} $-converges to $ v^{*} $. Since $ \Vert v^{*}_{l} - v^{*}_{l'} \Vert = 2 \varepsilon $ for $ l \neq l' $, any $ w^{*} $-open set $ U $ containing $ v^{*} $ fulfills $ \mathrm{diam}(U \cap B_{X^{*}}) \geq 2\varepsilon $. It follows that $ v^{*} \in (B_{X^{*}})'_{2\varepsilon} $ and, by our assumption, $ v^{*} \in (1 - \varepsilon) B_{X^{*}} $. At the same time,
$$ \Vert v^{*} \Vert \geq e^{**}_{n}(v^{*}) = (1 - \varepsilon) + \varepsilon e^{**}_{n}(u^{*}) > 1 - \varepsilon, $$
which is not possible.

Hence, the functionals $ e^{**}_{n} $ are $ w^{*} $-continuous indeed. Every $ e^{**}_{n} $ is therefore the evaluation of some $ e_{n} \in X $. Finally, it is easy to check that $ e_{1}, e_{2}, \dots $ is a basis of $ X $ that is $ 1 $-equivalent to the canonical basis of $ c_{0} $.
\end{proof}

\begin{proof}[Proof of Proposition~\ref{prop:classOfc0}]
Let $\F$ be the set of those $\mu\in\PP_\infty$ for which $X_\mu$ is an $ \mathcal{L}_{\infty, 1+} $-space. By Theorem~\ref{thm:scriptLp}, $\F$ is a $G_\delta$-set in $\PP_\infty$. Let $\Omega$ be the mapping from Lemma~\ref{lem:slozitost_koule} and let us denote by $\triangle$ the closed set $\{(x,x)\colon x\in\mathcal K(\ell_\infty)\}$ in $\mathcal K(\ell_\infty)\times \mathcal K(\ell_\infty)$. By Lemma~\ref{lem:ballIsOmega} and Theorem~\ref{thm:characterizationOfc0}, we have that
\[
\isomtrclass{c_0} = \F\cap\{\nu\in\PP\colon (\tfrac{1}{2}\Omega(\nu),\Omega'_1(\nu))\in\triangle\}.
\]
By Lemma~\ref{lem:slozitost_koule} and Lemma~\ref{lem:slozitost_derivace}, the mapping $\PP\ni \nu\mapsto (\tfrac{1}{2}\Omega(\nu),\Omega'_1(\nu))\in \mathcal K(\ell_\infty)\times \mathcal K(\ell_\infty)$ is an $\boldsymbol{\Sigma}_3^0$-measurable, so we obtain that $\isomtrclass{c_0}$ is an $F_{\sigma \delta}$-set in $\PP_\infty$.
\end{proof}

\begin{corollary}\label{cor:lowerBoundDerivative}
Let $\varepsilon>0$. Then the mapping $s_\varepsilon$ from Lemma~\ref{lem:slozitost_derivace} is not $\boldsymbol{\Sigma}_2^0$-measurable.
\end{corollary}
\begin{proof}
Otherwise, similarly as in the proof of Proposition~\ref{prop:classOfc0} we would prove that $\isomtrclass{c_0}$ is a $G_\delta$-set in $\PP_\infty$, which is not possible due to Theorem~\ref{thm:gurariiTypicalInPGeneralized}.
\end{proof}

\section{Miscellaneous}\label{sect:misc}

\subsection{Superreflexive spaces}

Recall that a map $f:M\to N$ between metric spaces is called a $C$-bilipschitz embedding if
\[
\forall x\neq y\in M:\; C^{-1}d_M(x,y)< d_N(f(x),f(y))< Cd_M(x,y).
\]
\begin{lemma}\label{lem:finiteMetricSpaceEmbeds}
Let $M$ be a finite metric space and $C>0$. The set $E(M,C)$ consisting of those $\mu\in\PP$ such that $M$ admits a $C$-bilipschitz embedding into $X_\mu$ is open in $\PP$.
\end{lemma}
\begin{proof}
Let $\mu\in E(M,C)$. Thus, there is a $C$-bilipschitz embedding $f:M\to X_\mu$. By perturbing the image of $f$ if necessary, we may without loss of generality assume that $f(M)\subseteq V$.

Consider $\varepsilon>0$ and the open neighborhood $U_\varepsilon$ of $\mu$ consisting of those $\mu'\in\PP$  for which $|\mu(f(x)-f(y)) - \mu'(f(x)-f(y))|<\varepsilon$ for every $x,y\in M$. Then $U_\varepsilon\subseteq E(M,C)$ for $\varepsilon>0$ small enough. Indeed, it suffices to choose $\varepsilon$ smaller than $$\min\{\min_{x\neq y\in M} C d_M(x,y)-\mu(f(x)-f(y)),\min_{x\neq y\in M} \mu(f(x)-f(y))-C^{-1} d_M(x,y) \}.$$ The easy verification is left to the reader. 
\end{proof}

\begin{proposition}\label{prop:sequenceOfFiniteMetricSpacesEmbed}
Let $(M_n)_{n\in\Nat}$ be a sequence of finite metric spaces and let $\mathcal{X}$ be the class of those Banach spaces $X$ for which there exists a constant $C$ such that for every $n\in\Nat$, $M_n$ admits a $C$-bilipschitz embedding into $X$.

Then $\F:=\{\mu\in\PP_\infty\setsep X_\mu \text{ is in }\mathcal{X}\}$ is a $G_{\delta \sigma}$-set in $\PP_\infty$.
\end{proposition}
\begin{proof}
Follows immediately from Lemma~\ref{lem:finiteMetricSpaceEmbeds}, because we have
\[\F = \PP_\infty\cap \bigcup_{C>0}\bigcap_{n\in\Nat} E(M_n,C).\]
\end{proof}

Bourgain in his seminal paper \cite{B86} found a sequence of finite metric spaces $(M_n)_{n\in\Nat}$ such that a separable Banach space is not superreflexive if and only if there exists a constant $C$ such that for every $n\in\Nat$, $M_n$ admits a $C$-bilipschitz embedding into $X$. We refer the interested reader to \cite[Section 9]{ostKniha} for some more related facts and results. Thus, combining this result with Proposition~\ref{prop:sequenceOfFiniteMetricSpacesEmbed} we obtain immediately the following.

\begin{theorem}\label{thm:superreflexive}
The class of all superreflexive spaces is an $F_{\sigma \delta}$-set in $\PP_\infty$.
\end{theorem}

A metric space $M$ is called \emph{locally finite} if it is uniformly discrete and  all balls in $M$ are finite sets (in particular, every such $M$ is at most countable).
Let us mention a result by Ostrovskii by which a locally finite metric space bilipschitz embeds into a Banach space $X$ if and only if all of its finite subsets admit uniformly bilipschitz embeddings into $X$, see \cite{Ost12} or \cite[Theorem 2.6]{ostKniha}. Thus, from Proposition~\ref{prop:sequenceOfFiniteMetricSpacesEmbed} we obtain also the following.

\begin{corollary}\label{cor:locFiniteMetricSpaces}
Let $M$ be a locally finite metric space. Then the set of those $\mu\in\PP_\infty$ for which $M$ admits a bilipschitz embedding into $X_\mu$ is a $G_{\delta \sigma}$-set in $\PP_\infty$.
\end{corollary}

It is well-known that many important classes of separable Banach spaces are not Borel. This concerns e.g. reflexive spaces, spaces with separable dual,  spaces containing $\ell_1$, spaces with the Radon-Nikodým property, spaces isomorphic to $L_p[0,1]$ for $p\in(1,2)\cup(2,\infty)$, or spaces isomorphic to $c_0$. We refer to \cite[page 130 and Corollary 3.3]{Bo02} and \cite[Theorem 1.1]{K19} for papers which contain the corresponding results and to the monograph \cite{dodosKniha} and the survey \cite{G17} for some more information. Thus, e.g. in combination with Corollary~\ref{cor:locFiniteMetricSpaces}, we see that none of those classes might be characterized as a class into which a given locally finite metric space bilipchitz embeds. Let us give an example of such a result which is related to \cite[Problem 12.5(b)]{MS17}. This is an elementary, but interesting application of the whole theory.

\begin{corollary}
There does not exist a locally finite metric space $M$ such that any separable Banach space $X$ is not reflexive if and only if $M$ admits a bilipschitz embeddings into $X$.
\end{corollary}

\begin{remark}
Let us draw attention of the reader once more to the remarkable paper \cite{MS17}, where the authors found a metric characterization of reflexivity even though such a condition is necessarily non Borel (as mentioned above).
\end{remark}

\subsection{Szlenk indices}\label{subsection:szlenkIndices}

In this subsection we give estimates on the Borel classes of spaces with Szlenk index less than or equal to a given ordinal number. Note that it is a result by Bossard, see \cite[Section 4]{Bo02}, that those sets are Borel and their Borel classes are unbounded. So our contribution here is that we provide certain quantitative estimates from above. Similarly, we give an estimate on the Borel class of spaces with summable Szlenk index, which is a quantitative improvement of the result mentioned in \cite[page 367]{G10}. Let us start with the corresponding definitions.
Let $X$ be a real Banach space and $K\subseteq X^*$ a $w^*$-compact set. Following \cite{L06}, for $\varepsilon>0$ we define $s_\varepsilon(K)$ as the Szlenk derivative of the set $K$ (see Subsection~\ref{subsection:szlenkDerivative}) and then we inductively define $s_\varepsilon^\alpha(K)$ for an ordinal $\alpha$ by $s_\varepsilon^{\alpha+1}(K):=s_\varepsilon(s_\varepsilon^\alpha(K))$ and $s_\varepsilon^\alpha(K):=\bigcap_{\beta<\alpha} s_\varepsilon^\beta(K)$ if $\alpha$ is a limit ordinal. Given a real Banach space $X$, $Sz(X,\varepsilon)$ is the least ordinal $\alpha$ such that $s_\varepsilon^\alpha(B_{X^*})=\emptyset$, if such an ordinal exists (otherwise we write $Sz(X,\varepsilon) = \infty$). The Szlenk index is defined by $Sz(X) = \sup_{\varepsilon > 0} Sz(X,\varepsilon)$.

Recall that for a separable infinite-dimensional Banach space $X$ the Szlenk index is either $\infty$ or $\omega^\alpha$ for some $\alpha\in [1,\omega_1)$, see \cite[Section 3]{L06}.

\begin{theorem}\label{thm:Szlenk}
Let $\alpha\in[1,\omega_1)$ be an ordinal. Then
\[
\{\mu\in\PP_\infty\colon Sz(X_{\mu})\leq \omega^\alpha\}
\]
is a $\boldsymbol{\Pi}_{\omega^\alpha + 1}^0$-set in $\PP_\infty$.
\end{theorem}
\begin{proof}
Using Lemma~\ref{lem:slozitost_derivace}, it is easy to prove by induction on $n$ that the mapping $\mathcal K(B_{\ell_\infty})\ni F\mapsto s_\varepsilon^n(F)\in \mathcal K(B_{\ell_\infty})$ is $\boldsymbol{\Sigma}_{2n+1}^0$-measurable for every $n\in\Nat$. Further, the mapping $\mathcal K(B_{\ell_\infty})\ni F\mapsto s_\varepsilon^\omega(F)\in \mathcal K(B_{\ell_\infty})$ is $\boldsymbol{\Sigma}_{\omega+1}^0$-measurable. Indeed, for every open $V\subseteq B_{\ell_\infty}$, by compactness argument, we have
\[
\{F\colon s_\varepsilon^\omega(F)\subseteq V\} = \bigcup_{n=1}^\infty \{F\colon s_\varepsilon^n(F)\subseteq V\}
\]
which is a $\boldsymbol{\Sigma}_{\omega}^0$-set, so by Lemma~\ref{lem:meritelnost} the mapping $s_\varepsilon^\omega$ is $\boldsymbol{\Sigma}_{\omega+1}^0$-measurable. Similarly, we prove by transfinite induction that $s_\varepsilon^\beta$ is $\boldsymbol{\Sigma}_{\beta+1}^0$-measurable whenever $\beta\in[\omega,\omega_1)$ is a limit ordinal.

Let $\Omega$ be the mapping from Lemma~\ref{lem:slozitost_koule}. Then by Lemma~\ref{lem:ballIsOmega} we have
\[\begin{split}
\{\mu\in\PP_\infty\colon Sz(X_{\mu})\leq \omega^\alpha\} & = \bigcap_{k\in\Nat} \{\mu\in\PP_\infty\colon Sz(X_{\mu},\tfrac{1}{k})\leq \omega^\alpha\}\\ & = \bigcap_{k\in\Nat} \{\mu\in\PP_\infty\colon s_{1/k}^{\omega^\alpha}(\Omega(\mu)) = \emptyset\},
\end{split}\]
which, by the above and Lemma~\ref{lem:slozitost_koule}, is the countable intersection of preimages of closed sets under $\boldsymbol{\Sigma}_{\omega^\alpha + 1}^0$-measurable mapping, so it is a $\boldsymbol{\Pi}_{\omega^\alpha + 1}^0$-set in $\PP_\infty$.
\end{proof}

Let us recall that a Banach space $X$ has a \emph{summable Szlenk index} if there is a constant $M$ such that for all positive $\varepsilon_1,\ldots,\varepsilon_n$ with $s_{\varepsilon_1}\ldots s_{\varepsilon_n}B_{X^*}\neq\emptyset$ we have $\sum_{i=1}^n \varepsilon_i\leq M$.

\begin{proposition}\label{prop:summable}
The set $\{\mu\in\PP_\infty\colon X_{\mu}\textnormal{ has a summable Szlenk index}\}$ is a $\boldsymbol{\Sigma}_{\omega+2}^0$-set in $\PP_\infty$.
\end{proposition}
\begin{proof}
Let $\Omega$ be the mapping from Lemma~\ref{lem:slozitost_koule}. It is easy to see that the set $\{\mu\in\PP_\infty\colon X_{\mu}\text{ has a summable Szlenk index}\}$ is equal to
\[
\bigcup_{M\in\Nat}\quad \bigcap_{\substack{\varepsilon_1,\ldots,\varepsilon_n\in\Rat_+\\ \sum_{i=1}^n \varepsilon_i > M}} \{\mu\in\PP_\infty\colon s_{\varepsilon_1}\ldots s_{\varepsilon_n}\Omega(\mu)=\emptyset\},
\]
which by Lemma~\ref{lem:slozitost_koule} and Lemma~\ref{lem:slozitost_derivace} is a $\boldsymbol{\Sigma}_{\omega+2}^0$-set in $\PP_\infty$.
\end{proof}

Finally, let us note that similarly one can of course estimate Borel complexity of various other classes of spaces related to Szlenk derivations, e.g. spaces with Szlenk power type at most $p$ etc.

\subsection{Spaces having Schauder basis-like structures}
It is an open problem whether the class of spaces with Schauder basis is a Borel set in $\B$ (see e.g. \cite[Problem 8]{dodosKniha}) and note that by the results from \cite[Section~3]{CDDK1} it does not matter whether we use the coding $SB(C([0,1]))$ or $\B$. However, it was proved by Ghawadrah that the class of spaces with $\pi$-property is Borel (actually, it is a $\boldsymbol{\Sigma}_6^0$-set in $\PP_\infty$ which follows immediately from \cite[Lemma 2.1]{Gh15}, see also \cite{Gh19}) and that the class of spaces with the bounded approximation property (BAP) is Borel (actually, it is a $\boldsymbol{\Sigma}_7^0$-set in $\PP_\infty$ which follows immediately from \cite[Lemma 2.1]{Gh17} and this estimate has recently been improved to a $\boldsymbol{\Sigma}_6^0$-set in any admissible topology, see \cite{Gh19}).

One is therefore led to the question of finding examples of Banach spaces having BAP but not the Schauder basis. Such an example was constructed by Szarek \cite{Sz87}. Actually, Szarek considered classes of separable spaces with \emph{local basis structure (LBS)} and \emph{local $ \Pi $-basis structure (L$\Pi$BS)} for which we have
\[
\text{basis}\implies \text{(L$\Pi$BS)} \implies \Big(\text{(LBS) and (BAP)}\Big) \implies \text{(BAP)}
\]
and he proved that the converse to the second and the third implication does not hold in general. The problem of whether the converse to the first implication holds seems to be open, see \cite[Problem 1.8]{Sz87}. In this subsection we prove that both (LBS) and (L$\Pi$BS) give rise to a Borel class of separable Banach spaces (we even compute an upper bound on their Borel complexities, see Theorem~\ref{thm:lbs}). Note that this result somehow builds a bridge between both open problems mentioned above, that is, between the problem of whether $\langle \text{spaces with Schauder basis} \rangle$ is a Borel set in $\B$ and the problem of whether every separable Banach space with (L$\Pi$BS) has a basis.

Let us start with the definitions as they are given in \cite{Sz87}.

\begin{definition}
By the \emph{basis constant} of a basis $ (x_{i})_{i=1}^{d} $ of a Banach space $ X $ of dimension $ d \in [0, \infty] $ we mean the least number $ C \geq 1 $ such that $ \Vert \sum_{i=1}^{n} a_{i} x_{i} \Vert \leq C \Vert \sum_{i=1}^{m} a_{i} x_{i} \Vert $ whenever $ n , m \in \mathbb{N} $, $ n \leq m \leq d $ and $ a_{1}, \dots, a_{m} \in \mathbb{R} $. The basis constant of $ (x_{i})_{i=1}^{d} $ is denoted by $ \mathrm{bc}( (x_{i})_{i=1}^{d} ) $. We further denote
$$ \mathrm{bc}(X) = \inf \big\{ \mathrm{bc}( (x_{i})_{i=1}^{d} ) : \textnormal{$ (x_{i})_{i=1}^{d} $ is a basis of $ X $} \big\}. $$
\end{definition}

\begin{definition}
A Banach space $ X $ is said to have the \emph{local basis structure (LBS)} if $ X = \overline{\bigcup_{n=1}^{\infty} E_{n}} $, where $ E_{1} \subseteq E_{2} \subseteq \dots $ are finite-dimensional subspaces satisfying $ \sup_{n \in \mathbb{N}} \mathrm{bc}(E_{n}) < \infty $.

Further, $ X $ is said to have the \emph{local $ \Pi $-basis structure (L$\Pi$BS)} if $ X = \overline{\bigcup_{n=1}^{\infty} E_{n}} $, where $ E_{1} \subseteq E_{2} \subseteq \dots $ are finite-dimensional subspaces satisfying $ \sup_{n \in \mathbb{N}} \mathrm{bc}(E_{n}) < \infty $ for which there are projections $ P_{n} : X \to E_{n} $ such that $ P_{n}(X) = E_{n} $ and $ \sup_{n \in \mathbb{N}} \Vert P_{n} \Vert < \infty $.
\end{definition}

\begin{lemma}\label{lem:eta}
Whenever $E$ is a finite-dimensional subspace of a Banach space $X$, $\delta\in (0,1)$, $K>0$, $T:E\to X$ is a $(1+\delta)$-isomorphism (not necessarily surjective) with $\|T-I\|<\delta$ and $P:X\to E$ is a projection with $P(X)=E$ and $\|P\|\leq K$, then for every subspace $F$ of $E$ we have $\|TP|_{T(F)} - I_{T(F)}\|\leq 4\delta K$. 

Moreover, whenever $\|TP|_{T(E)} - I_{T(E)}\|\leq q <1$ then $T(E)$ is $\frac{(1+\delta) K}{1-q}$-complemented in $X$.
\end{lemma}
\begin{proof}
Let $\{f_1,\ldots,f_n\}$ be a basis of $F$. Then for every $x=\sum_{i=1}^n a_i T(f_i)\in T(F)$ we have
\[
\|TPx-x\|=\|TP(\sum_{i=1}^n a_i(T(f_i)-f_i))\|\leq (1+\delta)K\|(T-I)T^{-1}x\|\leq (1+\delta)^2\delta K \|x\|.
\]
Moreover, if $\|TP|_{T(E)} - I_{T(E)}\|<1$ then the mapping $TP|_{T(E)}$ is an isomorphism with $\|(TP|_{T(E)})^{-1}\|\leq \sum_{i=0}^\infty q^i = \frac{1}{1-q}$. It is now straightforward to prove that $P':=(TP|_{T(E)})^{-1} TP:X\to T(E)$ is a projection onto $T(E)$ with $\|P'\|\leq \frac{(1+\delta) K}{1-q}$.
\end{proof}

\begin{lemma}\label{lem:lbsCharacterization}
For every $\mu\in\B$, $K,l\in\Nat$ and $v_1,\ldots,v_m\in V$, let us denote by $\Phi(\mu,K,v_1,\ldots,v_m)$ and $\Psi(\mu,K,l,v_1,\ldots,v_m)$ the formulae
$$\Phi(\mu,K,v_1,\ldots,v_m)=\forall a_{1}, \dots, a_{m} \in \mathbb{R} : \max_{1 \leq k \leq m} \mu \Big( \sum_{i=1}^{k} a_{i} v_{i} \Big) \leq K \mu \Big( \sum_{i=1}^{m} a_{i} v_{i} \Big)$$ and $$\Psi(\mu,K,l,v_1,\ldots,v_m)=\exists u_{1}, \dots, u_{l} \in \mathbb{Q}\textrm{-}\mathrm{span} \{ v_{1}, \dots, v_{m} \} \, \forall a_{1}, \dots, a_{m}, b_{1}, \dots, b_{l} \in \mathbb{R} : $$
$$ \mu \Big( \sum_{i=1}^{m} a_{i} v_{i} + \sum_{i=1}^{l} b_{i} u_{i} \Big) \leq K \mu \Big( 
\sum_{i=1}^{m} a_{i} v_{i} + \sum_{i=1}^{l} b_{i} e_{i} \Big).$$ 
Then for every $\nu\in\B$ the following holds.

\begin{enumerate}[(a)]
    \item The space $ X_{\nu} $ has LBS if and only if
\[ \exists K \in \mathbb{N} \, \forall n \in \mathbb{N} \, \exists m \in \mathbb{N} \, \exists v_{1}, \dots, v_{m} \in V, \, \{ e_{1}, \dots, e_{n} \} \subseteq \mathrm{span} \{ v_{1}, \dots, v_{m} \}\]
\[ \Phi(\nu,K,v_1,\ldots,v_m). \]
    \item The space $ X_{\nu} $ has L$\Pi$BS if and only if
$$ \exists K \in \mathbb{N} \, \forall n \in \mathbb{N} \, \exists m \in \mathbb{N} \, \exists v_{1}, \dots, v_{m} \in V, \, \{ e_{1}, \dots, e_{n} \} \subseteq \mathrm{span} \{ v_{1}, \dots, v_{m} \}$$
$$ \Phi(\nu,K,v_1,\ldots,v_m) \wedge \forall l\in\Nat \Psi(\nu,K,l,v_1,\ldots,v_m).$$
\end{enumerate}
\end{lemma}
\begin{proof}
We prove only the more difficult part (b). Since $ \nu \in \mathcal{B} $, the space $ X_{\nu} $ is just the completion of $ (c_{00}, \nu) $ (it is not necessary to consider a quotient). So, the notions of linear span and of linear independence have the same meaning in $ c_{00} $ and in $ X_{\nu} $, if performed on subsets of $ c_{00} $.

Let us suppose that $ \nu \in \mathcal{B} $ satisfies the formula in (b) for some $ K \in \mathbb{N} $. We put $ E_{0} = \{ 0 \} $ and choose recursively subspaces $ E_{1} \subseteq E_{2} \subseteq \dots $ of $ X_{\nu} $, each of which is generated by a finite number of elements of $ V $, in the following way. Assuming that $ E_{j} $ has been already chosen, we pick first $ n_{j+1} \geq j + 1 $ such that $ E_{j} \subseteq \mathrm{span} \{ e_{1}, \dots, e_{n_{j+1}} \} $. Then we can pick $ m_{j+1} \in \mathbb{N} $ and $ v^{j+1}_{1}, \dots, v^{j+1}_{m_{j+1}} \in V $ with $ \{ e_{1}, \dots, e_{n_{j+1}} \} \subseteq \mathrm{span} \{ v^{j+1}_{1}, \dots, v^{j+1}_{m_{j+1}} \} $ such that $\Phi(\nu,K,v^{j+1}_{1}, \ldots, v^{j+1}_{m_{j+1}})$ and for every $l\in\Nat$, $\Psi(\nu,K,l,v^{j+1}_{1}, \ldots, v^{j+1}_{m_{j+1}})$ hold.

We put
$E_{j+1} = \mathrm{span} \big\{ v^{j+1}_{1}, \dots, v^{j+1}_{m_{j+1}} \big\}$.
In this way, we obtain $ E_{j} \subseteq E_{j+1} $. Also, $ X_{\nu} = \overline{\bigcup_{n=1}^{\infty} E_{n}} $ (we have $ e_{j+1} \in E_{j+1} $, as $ n_{j+1} \geq j + 1 $). If we take all non-zero vectors $ v^{j+1}_{i}, 1 \leq i \leq m_{j+1} $, we obtain a basis of $ E_{j+1} $ with the basis constant at most $ K $.

To show that the sequence $ E_{1} \subseteq E_{2} \subseteq \dots $ witnesses that $ X_{\nu} $ has L$\Pi$BS, it remains to find a projection $ P_{j+1} $ of $ X_{\nu} $ onto $ E_{j+1} $ such that $ \Vert P_{j+1} \Vert \leq K $. Let us pick some $ l \in \mathbb{N} $ and put $ E^{(l)} = \mathrm{span} \{ v^{j+1}_{1}, \dots, v^{j+1}_{m_{j+1}}, e_{1}, \dots, e_{l} \} $. By $\Psi(\nu,K,l,v^{j+1}_{1}, \ldots, v^{j+1}_{m_{j+1}})$, there exists a projection $P^{(l)}$ of $E^{(l)}$ onto $E_{j+1}$ with  $\|P^{(l)}\|\leq K$. Since the norms of $P^{(l)}$, for $l\in\Nat$, are uniformly bounded and have a fixed finite-dimensional range, there exists their accumulation point in SOT which is a projection $P_{j+1}: X_\nu\rightarrow E_{j+1}$ of norm bounded by $K$ as desired.

Conversely, suppose that $X_\nu$ has L$\Pi$BS as witnessed by some $C>1$ and a sequence $(E_n)_{n\in\Nat}$ of finite-dimensional subspaces satisfying $X_\nu=\overline{\bigcup_n E_n}$ and $ \sup_{n \in \mathbb{N}} \mathrm{bc}(E_{n}) <C$, for which there are projections $ P_{n} : X_\nu \to E_{n} $ such that $ P_{n}(X_\nu) = E_{n} $ and $ \sup_{n \in \mathbb{N}} \Vert P_{n} \Vert < C$.
Pick $D>0$ such that $H_n:=(\Span\{e_1,\ldots,e_n\},\nu)$ is $D$-complemented in $X_\nu$ and let $\phi_1:=\phi_2^{e_1,\ldots,e_n}$ be the function from Lemma~\ref{lem:approx2}. Fix $\varepsilon>0$ such that $\phi_1(t)$ is small enough (to be specified later) whenever $t<\varepsilon$. Find $k\in\Nat$ such that there are $h_1,\ldots,h_n\in E_k$ with $\nu(e_i-h_i)<\varepsilon$. If $\phi_1(\varepsilon)$ is small enough we have $\frac{(1+\phi_1(\varepsilon))D}{1-4\phi_1(\varepsilon)D}\leq 2D$ (this value refers to the ``Moreover'' part in Lemma~\ref{lem:eta}). By Lemma~\ref{lem:eta}, $\Span\{h_i\colon i\leq n\}$ is $2D$-complemented in $X_\nu$, so let $Q:X_\nu\to \Span\{h_i\colon i\leq n\}$ be the corresponding projection. Pick a basis $h_{n+1},\ldots,h_{\dim E_k}$ of the space $E_k\cap Q^{-1}(0)$ which is $(2D+1)$-complemented in $E_k$.
Let $\phi_2:=\phi_2^{h_{n+1},\ldots,h_{\dim E_k}}$ be the function from Lemma~\ref{lem:approx2}. Fix $\delta>0$ such that $\phi_2(t)$ is small enough (to be specified later) whenever $t<\delta$. Finally, find $f_{n+1},\ldots,f_{\dim E_k}\in V$ with $\nu(f_j-h_j)<\delta$ for $j=n+1,\ldots,\dim E_k$.

We \emph{claim} that the space $F_n:=(\Span\{e_1,\ldots,e_n,f_{n+1},\ldots,f_{\dim E_k}\},\nu)$ is $2C$-complemented in $X_\nu$ and $d_{BM}(F_n,E_k)<2$. If we denote by $T:E_k\to F_n$ the linear mapping given by $h_i\mapsto e_i$, $i\leq n$, and $h_j\mapsto f_j$, $n+1\leq j\leq \dim E_k$, then for every $y\in \Span\{h_i\colon i\leq n\}$ and 
$z\in \Span\{h_j\colon j=n+1,\ldots,\dim E_k\}$ we have
\[\begin{split}
\nu(T(y+z)-y-z) & \leq \nu(Ty-y) + \nu(Tz-z)\leq \phi_1(\varepsilon)\nu(y) + \phi_2(\delta)\nu(z)\\
& \leq \Big(\phi_1(\varepsilon)2D + \phi_2(\delta)(2D+1)\Big)\nu(y+z);
\end{split}\]
hence, if $\eta:=\Big(\phi_1(\varepsilon)2D + \phi_2(\delta)(2D+1)\Big)<1$, we obtain $\|T\|\leq 1 + \|I-T\|\leq 1+\eta$ and $\|Tx\|\geq \|x\|-\|(I-T)x\|\geq (1-\eta)\|x\|$ for every $x\in E_k$ so $T$ is an isomorphism with $\|T\|^{-1}\leq (1-\eta)^{-1}$. Thus, by Lemma~\ref{lem:eta}, if $\phi_1(\varepsilon)$ and $\phi_2(\delta)$ are small enough (and so $\eta$ is small enough), we obtain $\|T\|\|T^{-1}\|<2$ and $F_n$ is $2C$-complemented in $X_\nu$.

Thus, $\mathrm{bc}(F_n)\leq \mathrm{bc}(E_k)d_{BM}(E_k,F_n)<2C$ which is witnessed by some basis $v_1,\ldots,v_m\in V$ of $F_n$. This shows that $\Phi(\nu,2C,v_1,\ldots,v_m)$ holds. Let $P:X_\nu\rightarrow F_n$ be a projection with $P[X_\nu]=F_n$ and $\|P\|\leq 2C$. Given $ l \in \Nat $, let $ T \subseteq \{ 1, \dots, l \} $ be a set such that $ (e_{i})_{i \in T} $ together with $ (v_{i})_{i=1}^{m} $ form a basis of $ \Span (\{ v_{1}, \dots, v_{m} \} \cup \{ e_{1}, \dots, e_{l} \}) $. Pick $A>0$ such that $(v_i)_{i=1}^m\cup (e_i)_{i\in T} \approximates[\ell_1^{m+|T|}]{A}$. 
For $i\in T$ pick $u_i\in \qeSpan \{v_1,\ldots,v_m\}$ such that $\nu(u_i-P(e_i))<\frac{C}{A}$. Then for every $a_1,\ldots,a_m\in\Rea$ and every $(b_i)_{i\in T}\in \Rea^T$ we have
\[\begin{split}
\nu \Big( \sum_{i=1}^{m} a_{i} v_{i} + \sum_{i\in T} b_{i} u_{i} \Big) & \leq 2C \nu \Big( 
\sum_{i=1}^{m} a_{i} v_{i} + \sum_{i\in T} b_{i} e_{i} \Big) + \nu(\sum_{i\in T} b_{i} (u_{i} - P(e_i)))\\ & \leq 3C \nu \Big( 
\sum_{i=1}^{m} a_{i} v_{i} + \sum_{i\in T} b_{i} e_{i} \Big).
\end{split}\]
Thus, the linear mapping $O:\Span (\{ v_{1}, \dots, v_{m} \} \cup \{ e_{1}, \dots, e_{l} \})\to \Span \{ v_{1}, \dots, v_{m} \}$ given by $v_i\mapsto v_i$, $i\leq m$, and $e_i\mapsto u_i$, $i\in T$, is a linear projection, and if we put $u_i:=O(e_i)\in V$ for every $i\in\{1,\ldots,l\}$, we see that $\Psi(\nu,3C,l,v_1,\ldots,v_m)$ holds and the formula in $(b)$ is satisfied with $K=3C$.
\end{proof}

\begin{theorem}\label{thm:lbs}
(a) The class of spaces which have LBS is a $ \boldsymbol{\Sigma}^{0}_{4} $-set in $ \mathcal{B} $.

(b) The class of spaces which have L$\Pi$BS is a $ \boldsymbol{\Sigma}^{0}_{6} $-set in $ \mathcal{B} $.
\end{theorem}

\begin{proof}
This follows from Lemma~\ref{lem:lbsCharacterization} because the conditions given by formulas $\Phi$ and $\Psi$ are obviously closed and $F_\sigma$, respectively.
\end{proof}

\section{Open questions and remarks}\label{sect:question}

In Theorem~\ref{thm:FSigmaIsomrfClasses} we proved that $\ell_2$ is the unique separable infinite-dimensional Banach space (up to isomorphism) whose isomorphism class is an $F_\sigma$-set. Following \cite{JLS96}, we say that a separable infinite-dimensional Banach space $X$ is \emph{determined by its finite dimensional subspaces} if it is isomorphic to every separable Banach space $Y$ which is finitely crudely representable in $X$ and for which $X$ is finitely crudely representable in $Y$. Note that $\ell_2$ is determined by its finite dimensional subspaces and that if a separable infinite-dimensional Banach space is determined by its finite dimensional subspaces then it is obviously determined by its pavings and so, by Theorem~\ref{thm:pavingDetermined2}, its isomorphism class is $G_{\delta\sigma}$. Johnson, Lindenstrauss, and Schechtman conjectured (see \cite[Conjecture 7.3]{JLS96}) that $\ell_2$ is the unique, up to isomorphism, separable infinite-dimensional Banach space which is determined by its finite dimensional subspaces. We believe that Theorem~\ref{thm:FSigmaIsomrfClasses} could be instrumental for proving this conjecture, since it follows from this theorem that the conjecture is equivalent to the positive answer to the following question. We thank Gilles Godefroy who suggested to us that there might be a relation between having $F_\sigma$ isomorphism class and being determined by finite dimensional subspaces.

\begin{question}\label{q:ellTwo}
Let $X$ be a separable infinite-dimensional Banach space determined by its finite dimensional subspaces. Is $\isomrfclass{X}$ an $F_\sigma$-set in $\B$?
\end{question}

It would be interesting to know whether there is a separable infinite-dimensional Banach space $X$ such that $\isomrfclass{X}$ is a $G_\delta$-set in $\B$ or in $\PP_\infty$. Note that the only possible candidate is the Gurari\u{\i} space, see Section~\ref{subsec:smallIsometry} for more details. One of the possible strategies to answer Question~\ref{q:4} in negative for $\PP_\infty$ would be to find an admissible topology $\tau$ on $SB(X)$ such that $\isomrfclass{\mathbb{G}}$ is a dense and meager set in $(SB(X),\tau)$. However, we do not even know whether $\isomrfclass{\mathbb{G}}$ is Borel.

\begin{question}\label{q:4}
Is $\isomrfclass{\mathbb{G}}$ a $G_\delta$-set in $\PP_\infty$ or in $\B$? Is it at least Borel?
\end{question}

Solving the homogeneous Banach space problem, Komorowski and Tomczak-Jaegermann (\cite{KoTJ95}), and Gowers (\cite{G02}) proved that if a separable infinite-dimensional Banach space is isomorphic to all of its closed infinite-dimensional subspaces, then it is isomorphic to $\ell_2$. It seems that the isometric variant of this result is open; that is, whether $\ell_2$ is the only (up to isometry) separable infinite-dimensional Banach space that is isometric to all of its infinite-dimensional closed subspaces. We note that any Banach space satisfying this criterion must be, by the Gowers' result, isomorphic to $\ell_2$. Our initial interest in this problem was that we observed that a positive answer implies that whenever $\isomtrclass{X}$ is closed in $\PP_\infty$ then $X\equiv\ell_2$. Eventually we found another argument (see Subsection~\ref{subsec:smallIsometry}), but the question is clearly of independent interest.

\begin{question}\label{q:noe}
Let $X$ be a separable infinite-dimensional Banach space which is isometric to all of its closed infinite-dimensional subspaces. Is then $X$ isometric to $\ell_2$?
\end{question}

We note here that Question~\ref{q:noe} was already attacked by N. de Rancourt \cite{Noe20}, who was able to prove that if $X$ is as above (that is, isometric to all of its closed infinite-dimensional subspaces) then, for every $\varepsilon>0$, $X$ admits a $(1+\varepsilon)$-unconditional basis.

In Theorem~\ref{thm:Intro2} we proved that $\isomtrclass{\mathbb{G}}$, resp. $\isomtrclass{L_p[0,1]}$, for $p\in[1,\infty)$, are $G_\delta$-sets; we even proved that they are dense $G_\delta$-sets in $\PP_\infty$, resp. in $\mathcal{L}_{p,1+}\cap \PP_\infty$. Coincidentally, all these spaces are Fra\" iss\' e limits (we refer to \cite[Proposition 3.7]{FLMT19} for this statement about $L_p[0,1]$. According to \cite{FLMT19}, no other examples of separable Banach spaces which are Fra\"{\i}ss\'e limits seem to be known. We remark that a characterization of separable Banach spaces with $G_\delta$ isometry classes has been obtained in \cite{CDR}, where some new examples are presented.

It also follows that for $1\leq p<\infty$, $L_p[0,1]$ is a generic $\mathcal{L}_{p,1+}$-space. On the other hand, by Corollary~\ref{cor:LpnotQSLpgeneric}, for $p\in \left[1,2\right)\cup \left(2,\infty\right)$, $L_p[0,1]$ is not a generic $QSL_p$-space. For $p=2$, $\ell_2$ is obviously the generic $QSL_2$-space, and since $QSL_1$-spaces coincide with the class of all Banach spaces, for $p=1$, $\mathbb{G}$ is the generic $QSL_1$-space. This leaves open the next question.
\begin{question}
For $p\in \left(1,2\right)\cup \left(2,\infty\right)$, does there exist a generic $QSL_p$-space in $\B$ or $\PP_\infty$?
\end{question}

In Theorem~\ref{thm:superreflexive}, we have computed that the class of superreflexive spaces is an $F_{\sigma\delta}$-set. It is easy to check that the class of superreflexive spaces is dense in $\PP_\infty$ and $\B$, so it cannot be a $G_\delta$-set as then this class would have a non-empty intersection with the isometry class of $\Gurarii$ which is not superreflexive. However, the following is not known to us.
\begin{question}
Is the class of all superreflexive spaces $F_{\sigma\delta}$-complete in $\PP_\infty$ or $\B$?
\end{question}

Taking into account that spaces with a summable Szlenk index form a class of spaces which is a $\boldsymbol{\Sigma}^0_{\omega+2}$-set, see Proposition~\ref{prop:summable}, the following seems to be an interesting problem.

\begin{question}
Is the set $\{\mu\in\PP_\infty\colon X_{\mu}\textnormal{ has a summable Szlenk index}\}$ of a finite Borel class?
\end{question}

Even though we do not formulate it as a numbered question, a natural project to consider is to determine at least upper bounds for isometry classes of other (classical or less classical) separable infinite-dimensional Banach spaces, such as $C[0,1]$, $C([0,\alpha])$ with $\alpha$ countable ordinal, Orlicz sequence spaces, Orlicz function spaces, spaces of absolutely continuous functions, Tsirelson's space, etc.

Kechris in \cite[page 189]{KechrisBook} mentions that there are not known any natural examples of Borel sets from topology or analysis that are $\boldsymbol{\Pi}^0_\xi$ or $\boldsymbol{\Sigma}^0_\xi$, for $\xi\geq 5$, and not of lower complexity. We think that the area of research investigated in this paper is a good one to find such examples.

\bigskip

\noindent{\bf Acknowledgements.}
We would like to thank William B. Johnson for his helpful advice concerning the structure of subspaces and quotients of $L_p$ spaces (see Section~\ref{sect:simpleClasses}), Gilles Godefroy for  his suggestion to have a look at the paper \cite{JLS96}, and Ghadeer Ghawadrah for sending us a preliminary version of her paper \cite{Gh19}.

\bibliographystyle{siam}
\bibliography{ref}
\end{document}